\documentclass[a4paper,11pt]{amsart}
\usepackage{mathpazo}
\usepackage[utf8]{inputenc}
\usepackage{indentfirst, amsfonts, amsmath, amsthm, amssymb, amscd}
\usepackage{amsmath,amsfonts,amscd,bezier}
\usepackage{mathrsfs}
\usepackage{graphicx}
\usepackage{setspace}
\usepackage{color}
\usepackage{comment}
\usepackage{verbatim}

\usepackage{hyperref}
\usepackage[normalem]{ulem}

\usepackage{thmtools}
\usepackage{thm-restate}

\newtheorem{maintheorem}{Theorem}

\newtheorem{theorem}{Theorem}[section]

\newtheorem{corollary}[theorem]{Corollary}
\newtheorem{proposition}[theorem]{Proposition}
\newtheorem{lemma}[theorem]{Lemma}
\newtheorem{definition}[theorem]{Definition}

\theoremstyle{remark}
\newtheorem{remark}[theorem]{Remark}

\newtheorem{question}{Question}

\title[Classification of conditional measures]{Classification of conditional measures along certain invariant one-dimensional foliations}

\author{M. Espitia}
\address{Departamento de Matem\'atica, Estat\'istica e Computa\c c\~ao Cient\'ifica,
	IMECC-UNICAMP, Campinas-SP, Brazil.}
\email{maresno@gmail.com}

\author{G. Ponce}
\address{Departamento de Matem\'atica, Estat\'istica e Computa\c c\~ao Cient\'ifica,
	IMECC-UNICAMP, Campinas-SP, Brazil.}
\email{gaponce@unicamp.br}

\author{R. Var\~{a}o} 
\address{Departamento de Matem\'atica, Estat\'istica e Computa\c c\~ao Cient\'ifica,
	IMECC-UNICAMP, Campinas-SP, Brazil.}
\email{varao@unicamp.br}

\begin{document}
	\onehalfspace
	\maketitle
	
	\begin{abstract} 
Let $f:M\to M$ be a homeomorphism over a compact Riemannian manifold, ergodic with respect to a measure $\mu$ defined on the completion of the Borel $\sigma$-algebra and  $\mathcal F$  a $f$-invariant one dimensional continuous foliation of $M$ by $C^1$-leaves. Then, if $f$ preserves a continuous $\mathcal{F}$-arc length system, then we only have three possibilities for the conditional measures of $\mu$ along $\mathcal F$, namely:
		\begin{itemize}
		    \item they are atomic for almost every leaf, or
		      \item for almost every leaf they are equivalent to the measure $\lambda_x$ induced by the invariant arc-length system over $\mathcal F$, or
              \item for almost every leaf their support is a nowhere dense, perfect subset of the leaf.
		\end{itemize}

Furthermore, we show that restricted to ergodic partially hyperbolic diffeomorphism with one-dimensional topological neutral center direction, we are able to eliminate the third case obtaining a dichotomy.

	\end{abstract}

	\setcounter{tocdepth}{1}
	\tableofcontents

	\section{Introduction}
	
	
	Given a topological space $X$, $\mathcal B$ its Borel sigma-algebra and $\mu$ a probability measure on $X$, it is well known that for any sub-sigma algebra $\mathcal E\subset \mathcal B$, $\mu$ may be disintegrated over the partition induced by $\mathcal E$ in the sense that we may find a system of probability measures $\{\mu_x\}_{x\in X}$, such that $\mu_{x}\left(\bigcap_{x\in E, E\in \mathcal E}E \right) =1$, $x\mapsto \mu_x$ is Borel measurable and $\mu(B) = \int \mu_x(B) d\mu$ for any $B\in \mathcal B$. 
	This system is referred to as being the disintegration of $\mu$ along $\mathcal E$. This fact was latter extended to the more general context of Lebesgue spaces, where V. Rokhlin (see \cite{Ro52}) proved that $\mu$ may be disintegrated over any measurable partition, that is, over any partition that is countably generated by elements of the $\sigma$-algebra of the Lebesgue space in consideration.
	
	The mentioned theorem of Rokhlin is extensively used in dynamical systems and was further extended to more general contexts (see for example \cite{Simmons}, \cite{PossobonRodrigues}). In general, given a certain dynamics, in many contexts this dynamics admits certain invariant partitions which are dynamically defined and the study of the disintegration of the invariant measures along these partitions usually yields important properties of the dynamics. In some cases these partitions are naturally given by invariant foliations of the dynamics.
	
	In smooth ergodic theory for example, Anosov and partially hyperbolic diffeomorphisms admit a pair of invariant foliations called stable and unstable foliations, which we will denote here by $\mathcal F^s$ and $\mathcal F^u$. In these contexts, measure disintegration techniques have been an essential tool to obtain ergodicity and rigidity of some properties such as regular conjugacy between certain $C^1$-close Anosov maps. 
	
	In his seminal work D. Anosov proved \cite{Anosov1} that the disintegration of the volume measure along the unstable (resp. stable) foliation of a volume preserving Anosov diffeomorphism is absolutely continuous with respect to the leaf measure. This result was generalized to the stable and unstable foliation of partially hyperbolic diffeomorphisms(see \cite{BP} for example). This is clearly not the general case for an arbitrary foliation, an example for which absolute continuity does not occur was given by A. Katok \cite{Mi} and there are now several examples of other natures in the literature. More specifically,  Katok's example shows a foliation by analytic curves of $(0,1)\times \mathbb R / \mathbb Z$ such that there exists a full Lebesgue measure set which intersects each leaf in exactly one point. In this case we say that the conditional measures along the leaves are atomic, or that the foliation is atomic with respect to the referrence measure - which in the example of Katok is the standard two dimensional Lebesgue measure.
	
	Atomicity and absolute continuity are two extremes among the possibilities that one could expect when studying the conditional measures along a foliation. The first one is equivalent to saying that the conditional measures are atomic measures and the last one implies that the conditional measures are absolutely continuous with respect to the Riemannian measure of the leaf, a property which is usually called leafwise absolute continuity or Lebesgue disintegration of measure. Although these are two extreme behaviors among, a priori, many possibilities for the disintegration of a measure, recent results have indicated that this dichotomy is more frequent than one would at first expect.
	In \cite{RW} D. Ruelle and A. Wilkinson proved that for certain skew product type of partially hyperbolic dynamics, if the fiberwise Lyapunov exponent is negative then the disintegration of the preserved measure along the fibers is atomic. Later A. Homburg \cite{homburg.atomic} proved that for some examples treated in \cite{RW} one can actually prove that the disintegration is composed by only one atom per leaf. A. Avila, M. Viana and A. Wilkinson \cite{AVW} proved that for $C^1$- volume preserving perturbations of the time-$1$ map of geodesic flows on negatively curved surfaces, the disintegration of the volume measure along the center foliation is either atomic or absolutely continuous and that in the latter case the perturbation should be itself the time-$1$ map of an Anosov flow. Although inside the class of derived from Anosov diffeomorphisms non-Lebesgue and non-atomic disintegration was given in \cite{Va},
 G. Ponce, A. Tahzibi and R. Var\~ao \cite{PTV} exhibited an open class of volume preserving diffeomorphisms which have (mono) atomic disintegration along the center foliation and, recently, A. Tahzibi and J. Zhang \cite{TahzibiZhang} proved that non-hyperbolic measures of derived from Anosov diffeomorphisms on $\mathbb T^3$ also must have atomic disintegration along the center foliation, answering a question from \cite{PT}.

	
	
	In this paper our main goal is to better understand the disintegration of an invariant measure along an invariant foliation for the dynamics without requiring hyperbolicity or partial hyperbolicity for $f$ but assuming that the invariant foliation has some type of metric rigidity with respect to $f$. In other words, we aim to investigate what are the possible characterizations of the conditional measures obtained when we disintegrate $\mu$ over a foliation $\mathcal F$, assuming that the behavior of $f$ along $\mathcal F$ is very far from being hyperbolic.

    	\subsection{Statement of results}

        Before stating our main result we recall that a topological space $X$ is called a \textit{Cantor space} if it is homeomorphic to the standard ternary Cantor set. If $X$ is a Cantor space with the topology induced by a certain space $N\supset X$, we will also say that $X$ is a Cantor set or a Cantor subset of $N$.
        
        It is well known that a non-empty metric space $X$ is a Cantor space if, and only if, it is perfect, compact and totally disconnected. We will say that a topological space $X$ is a \textit{$\sigma$-Cantor space} if there is a sequence of nested Cantor spaces $K_1\subset K_2 \subset \ldots $ such that $X = \bigcup_{i=1}^{\infty} K_i$. Analogously, for $X\subset N$ we will say that $X$ is a $\sigma$-Cantor set or $\sigma$-Cantor subset of $N$ if it is a $\sigma$-Cantor space with the topology induced from $N$.

       In particular, observe that in the case where $X\subset \mathbb R$, $X$ is a $\sigma$-Cantor subset if and only if it is perfect and with empty interior, since the only connected sets in $\mathbb R$ are intervals.

       Given a homeomorphism $f:M \to M$ over a compact Riemannian manifold $M$ and $\mathcal F$ a $f$-invariant one dimensional continuous foliation of M whose leaves are $C^1$ immersed submanifolds of $M$, a \textit{$\mathcal F$-arc length system} is a continuous system of functions $\{l_x\}_{x\in M}$ where each $l_x$ is a function defined on the the equivalence classes of simple arcs of $\mathcal F(x)$ associating to each class a non-negative real number in a ``metric like fashion'', that is, it vanishes only on degenerate arcs and is additive. Furthermore this system is $f$-invariant in the sense that  $l_x \circ f = l_{f(x)}$. The formal definition is presented in Definition \ref{Flengthsystem}.
	
	The following is our first main result.
	
	\begin{restatable}{maintheorem}{introduction}\label{theorem:A}
		Let $f:M\to M$ be a  homeomorphism over a compact Riemannian manifold, $\mathcal F$  be a $f$-invariant one dimensional continuous foliation of $M$ by $C^1$-leaves and $\{l_x\}$ a $\mathcal{F}$-arc length system. If $f$ is ergodic with respect to a $f$-invariant probability measure $\mu$ then one of the following holds:
		\begin{itemize}
			\item[a)] the disintegration of $\mu$ along  $\mathcal F$ is atomic.
			\item[b)] for almost every $x\in M$, the conditional measure on $\mathcal F(x)$ is equivalent to the measure $\lambda_{x}$ defined on simple arcs of $\mathcal F(x)$ by: \[\lambda_{x}(\gamma([0,1]))=l_{x}(\gamma), \; \text{where} \; \gamma \; \text{is a simple arc.}\]
			\item[c)] for almost every $x\in M$, the conditional measure on $\mathcal F(x)$ is supported in a $\sigma$-Cantor subset of $\mathcal F(x)$. Equivalently, in the leaf topology this set is perfect and with empty interior. 
		\end{itemize}
	\end{restatable}

	The existence of invariant systems of metrics was obtained in \cite{BonattiZhang} for the context of transitive partially hyperbolic diffeomorphisms with topological neutral center, meaning that $f$ and $f^{-1}$ have Lyapunov stable center direction (see \cite[section 7.3.1]{HHUSurvey}), i.e, given any $\varepsilon>0$ there exists $\delta>0$ for which, given any $C^1$ path $\gamma$ tangent to the center direction, one has
	\[\operatorname{length}(\gamma) <\delta \Rightarrow \operatorname{length}(f^n(\gamma)) <\varepsilon, \quad \forall n\in \mathbb Z. \] 
	For these diffeomorphisms, the center direction integrates to a continuous foliation $\mathcal F^c$ of $M$ ( \cite[Corollary $7.6$]{HHUSurvey}). In this context we are able to improve the previous result.
    
	\begin{maintheorem}\label{theorem:casobonatti}
		Let $f:M \to M$ be a transitive $C^1$ partially hyperbolic diffeomorphism with one-dimensional topological neutral center direction over a compact Riemannian manifold. If $f$ is ergodic with respect to a $f$-invariant probability measure $\mu$ with full support then one of the following holds:
		\begin{itemize}
			\item[a)] the disintegration of $\mu$ along  $\mathcal F^c$ is atomic.
			\item[b)] for almost every $x\in M$, the conditional measure on $\mathcal F^c(x)$ is equivalent to the measure $\lambda_{x}$ defined on simple arcs of $\mathcal F^c(x)$ by: \[\lambda_{x}(\gamma([0,1]))=l_{x}(\gamma), \; \text{where} \; \gamma \; \text{is a simple arc.}\]
		\end{itemize}
	\end{maintheorem}

	\subsection{Organization of the paper}
	
	In Section \ref{sec:preliminaries} we give some preliminaries on measure theory and disintegration of measures along a foliation. In Section \ref{sec:BorelLaminations} we introduce the definition of $\mathcal{F}$-arc length system with respect to a dynamics in $(M,\mathcal{A},\mu)$ and the construction of the measures induced by this system. In Section \ref{sec:FiberedSpaces} we prove several technical lemmas concerning the continuity/measurability of certain functions such as the evaluation of conditional measures on certain balls inside the leaves of $\mathcal F$. Finally, in Sections \ref{sec:DDL} and \ref{sec:dicotomia} we prove Theorem \ref{theorem:A} and Theorem \ref{theorem:casobonatti} respectively.

	\section{Basics on conditional measures} \label{sec:preliminaries}
	All along the paper $(M,\mathcal{A},\mu)$ will be a  probability space, where $M$ is a compact Riemannian manifold, with dimension at least two, $\mu$ is a non-atomic Borel measure and $\mathcal{A}$ is a completion of the Borel $\sigma$-algebra $\mathcal{B}$ of $M$ with respect to the measure $\mu$.\footnote{ We note that for simplicity we are using a slight abuse of notation by also denoting the extension of the measure $\mu$ on $\mathcal A$ as $\mu$.}. In other words, $(M,\mathcal{A},\mu)$ is measurably isomorphic to $([0,1], \mathcal A_{[0,1]}, \text{Leb}_{[0,1]})$ where $\text{Leb}_{[0,1]}$ is the  standard Lebesgue measure on $[0,1]$ and $\mathcal A_{[0,1]}$ is the $\sigma$-algebra of Lebesgue measurable sets of $[0,1]$. In particular, whenever we state that a certain subset of $M$ is measurable without mentioning the $\sigma$-algebra, it is implicitly understood that it is $\mathcal A$-measurable. We will denote by $\mu(\cdot |U)$ the restriction of $\mu$ to a subset $U\subset M$, that is, it denotes the measure given by: $\mu(\cdot |U)=\mu(U)^{-1} \cdot \mu(B\cap U)$. \\

	Given a sub-$\sigma$-algebra $\mathcal{E}\subset\mathcal{B}$ generated by a
	countable family $\{E_n\}_{n\in \mathbb N}$,  the \textit{atom} of $x$ is the set given by
	\[[x]:=\bigcap_{E\in \mathcal E: \; x\in E} E.\]
	Since $\{E_n\}_{n\in \mathbb N}$ generates $\mathcal E$ we may also write
	\[[x]=\bigcap_{x\in E_n} E_n.\]
	Consequently, $[x]$ is a Borel set for every $x\in M$ and $\{[x] : x\in M\}$ is a partition of $M$.
	
	
	Given $\mathcal{E}\subset \mathcal{B}$ a countably generated sub-$\sigma$-algebra, a family of measures $\{\mu_x\}_{x\in M}$ is called a \textit{system of conditional measures of $\mu$ associated to $\mathcal{E}$} if
	\begin{enumerate}
		\item[i)] the function $x\mapsto \mu_x(B)$ is $\mathcal E$-measurable in the sense that for every $B\in \mathcal{B}$, $x\to \mu_x(B)$ is $\mathcal{E}$-measurable,
		\item[ii)] for $\mu$ almost every $x\in X$, $\mu_x([x])=1$,
		\item[iii)] $\mu(B)=\int_{y\in B} \mu_y(B)d\mu(y)$.
	\end{enumerate}
	%
	As it is well known, every Borel measure $\mu$ admits a system of conditional measures with respect to any countably generated sub-$\sigma$ algebra $\mathcal E \subset \mathcal B$ and such system is essentially unique, see for example \cite[Theorem 5.14]{Einsiedler}.
	
	In our context it is usually convenient to consider systems of conditional measures along certain partitions by $C^1$ immersed submanifolds, as we detail in the sequel.
	
	We say that $\mathcal{F}$ is a \textit{continuous foliation of dimension $m$ by $C^1$-leaves} if  $\mathcal F = \{ \mathcal F(x)\}_{x \in M}$  is a partition of $M$  into $C^1$ immersed submanifolds of dimension $m$, such that for every $x\in M$ there exist open sets $U\subset M, \, V\subset \mathbb{R}^m, \, W\subset \mathbb{R}^{n-m}$ and  a homeomorphism $\varphi:U\to V\times W$, called a \textit{local chart}, such that for every $c\in W$ the set $\varphi^{-1}(V\times \{c\})$, which is called a \textit{plaque} of $\mathcal{F}$  is a connected component of $L\cap U$ for a certain $L\in \mathcal{F}$.  Given a foliation $\mathcal{F}$ of $M$, we denote by $\mathcal{F}(x)$ the element of $\mathcal{F}$ which contains $x$ and call such elements the \textit{leaves} of $\mathcal{F}$. 
	Whenever $U$ is a foliated chart, we denote by $\mathcal F|U$ the continuous foliation of $U$ given by plaques of $\mathcal F$ restricted to $U$.
	
	Given a one-dimensional continuous foliation $\mathcal F$ of $M$, consider a local chart of $\mathcal F$
	\[\varphi: U \to (0,1)\times B^{n-1}_1(0),\]
	we also call $U$ a foliated box. 
	
	Let $\{\widetilde{E}_{q,k}\}$ a countable collection of sets defined by 
	\[
	\widetilde{E}_{q,k}=(0,1)\times B(q,1/k)\subset\mathbb{R}\times \mathbb{R}^{n-1}, \quad q\in B^{n-1}_1(0)\cap\mathbb{Q}^{n-1}, \, k\in\mathbb{N}.
	\]
	Let $E_{q,k}=\varphi^{-1}(\widetilde{E}_{q,k})$ and $\mathcal{E}\subset\mathcal{B}$ the sub-$\sigma$-algebra generated by the family of Borelian sets $E_{q,k}$. Notice that for every $y\in M$ the atom, $[y]$ is the connected component of $\mathcal{F}(y)\cap U$ that contains $y$, in this case we also call the system of conditional measures of $\mu$ associated to $\mathcal E$, $\{\mu^U_y\}_{y\in U}$, \textit{the disintegration of $\mu$ along $\mathcal{F}$ restricted to $U$}. 
	We say that the disintegration of $\mu$ along $\mathcal F$ is \textit{atomic} if for any local chart $U$, for almost every $y\in U$ there exists $a(y)$ in the plaque $\mathcal F|U(y)$ with $\mu^U_y(a(y)) >0$.
	
	
	As we may observe, the disintegration of $\mu$ along $\mathcal F$ is always done inside a local chart. However, the following well known result states that if two local charts intersect, the disintegration of both local charts are equal except by a multiplicative constant when restricted to the intersection.
	
	\begin{proposition}(see for example \cite[Proposition 5.17]{El.Pisa}) \label{prop:disintegration.unbounded}
		If $U_1$ and $U_2$ are domains of two local charts $\varphi_{1}$ and $\varphi_{2}$ of $\mathcal F$, then for almost every $x$ the conditional measures $\mu_x^{U_1}$ and $\mu_x^{U_2}$ coincide up to a multiplicative constant on $U_1 \cap U_2$.
	\end{proposition}


    The above proposition provides a tool to extend a conditional measure (which is locally defined) to a measure on the full leaf that, when normalized on foliated charts, restores the disintegration. That is, assume $x$ is a point that belongs to two different foliated charts $U_1$ and $U_2$, and let $L_1$ and $L_2$ be the respective plaques that contain the point $x$. Since by the above proposition $\mu_x^{U_1}=\alpha \mu_x^{U_2}$ at the intersection of both charts, we can extend $\mu_x^{U_1}$ to $L_1 \cup L_2$ in such a way that on $L_2$ the extension of $\mu_x^{U_1}$ is $\alpha \mu_x^{U_2}$. This process can be done inductively so that $\mu_x^{U_1}$ is extended to the whole leaf $\mathcal F(x)$. Therefore, $\mathcal F(x)$ have many measures which are constructed in such a way but using different plaques to start from, and there is no canonical way to choose one of them, but this is not a problem as we work with a class of measures instead.
    
    As observed in \cite{AVW}, the property such as in Proposition \ref{prop:disintegration.unbounded}, allows us to define a family of classes of measures $\{x \in M: \Omega_x\}$, such that
	\begin{itemize}
	\item $\omega_x (M\setminus \mathcal F(x))=0$, for any representative $\omega_x$ of $\Omega_x$,
	    \item any two representatives of $\Omega_x$ are equal modulo multiplication by a constant
	    \item for any local chart $U$ of $M$ and $\{\mu^U_x\}_{x\in U}$ a disintegration of $\mu(\cdot |U)$ along $\mathcal F|U$, for almost every point $x\in U$ we have
	    \[\mu^U_x = \omega_x(\cdot | \mathcal F|U(x)),\]
	    where $\omega_x$ denotes a representative of $\Omega_x$.

	\end{itemize}

	The family $\{\Omega_x\}_{x\in M}$ will be called a disintegration of $\mu$ along $\mathcal F$. Notice that considering such a family $\{\Omega_x\}_{x\in M}$ an atom in a certain leaf $\mathcal F(x)$ is now a well-defined concept independent of the local chart since it is a point $a\in \mathcal F(x)$ for which $\Omega_x(\{a\})>0$. Hence, in the context of ergodic system, since the set of atoms is invariant is has zero or full measure, which implies that, in the later case almost every measure of the family will be a sum of Dirac measures.

	\section{Invariant arc-length and invariant metric systems} \label{sec:BorelLaminations}
	

	Given $f:M\to M$ and a foliation $\mathcal F$ of $M$, we say that $f$ \textit{preserves} $\mathcal F$, or that $\mathcal F$ \textit{is $f$-invariant} if for $x\in M$
	\[\mathcal F(f(x)) = f(\mathcal F(x)).\]
	From now on, $\mathcal{F}$ will denote a continuous and $f$-invariant one dimensional foliation.
	
	By a \textit{simple arc} $\gamma$ on a leaf $\mathcal F(x)$, we mean a $C^1$ curve $\gamma:[0,1]\to \mathcal F(x)$ for which $\gamma(t) \ne \gamma(s)$ for all $t\ne s$ with $(t,s)\notin \{(0,1), \; (1,0)\}$. In the space of simple arcs we define an equivalence relation by saying that $\gamma \sim \sigma$ if $\sigma$ is a reparametrization of $\gamma$. We say that a sequence of simple arcs $\gamma_n$ \textit{converges} to $\gamma$ (in the $C^0$-topology) if $\gamma_n$ converges pointwise to $\gamma$. By convention, by a \textit{degenerate arc} we mean a point.
	
	The following definition is inspired by the concept of center metric given in \cite{BonattiZhang}. 
	
	\begin{definition}\label{Flengthsystem}
		We will call $\{l_x\}$ a $\mathcal{F}$-arc length system, if for $x\in M$,  $l_x$ is defined on the simple arcs on $\mathcal{F}(x)$,  and  $l_x$ satisfies the following properties:
		\begin{enumerate}
			\item strictly positive on the non-degenerate arcs, and vanishing on degenerate arcs,
			\item $l_x(\gamma)=l_x(\sigma)$ if $\gamma \sim \sigma$,
			\item let $\gamma:[0,1]\to \mathcal{F}(x)$ be a simple arc and $a\in (0,1)$ then 
			\[
			l_x(\gamma[0,a])+l_x(\gamma[a,1])=l_x(\gamma[0,1]),
			\]
			\item let $\gamma:[0,1]\to \mathcal{F}(x)$ a simple arc, then 
			\[l_x(\gamma[0,1])=l_{f(x)}(f(\gamma[0,1])).\]
			\item given a sequence of simple arcs $\gamma_n :[0,1]\to \mathcal F(x_n)$, converging to a simple arc $\gamma:[0,1] \to \mathcal F(x)$,
\[l_{x_n}(\gamma_n) \rightarrow l_x(\gamma), \quad \text{as} \; n\rightarrow +\infty.\]	
		\end{enumerate}	
	\end{definition}
	
	In general, it is easy to give examples of systems preserving some continuous foliation of  dimension one $\mathcal F$ and admitting some $\mathcal F$-arc length system. \quad \\
	
	{\bf Examples:} \\
	\begin{itemize}
	\item[a)] Let $M=\mathbb T^d$, $d\geq 2$, and $L:\mathbb R^d \to \mathbb R^d$ a linear map given by a matrix with integer entries and for which $1$ is an eigenvalue. Let $v$ be an eigenvector associated to $1$ and take $E = \mathbb R\cdot v$. The linear map $L$ induces a linear function $f_L:\mathbb T^d \to \mathbb T^d$ and $E$ induced a foliation $\mathcal F$ of $\mathbb T^d$ which is one-dimensional and $f_L$-invariant. Clearly $f_L$ is an isometry along $\mathcal F$. In particular, the family of standard arc-lengths on the leaves of $\mathcal F$ constitute a $\mathcal F$-arc length system.
	For this case we obtain that any ergodic measure invariant by $f_L$ must have either atomic conditionals, conditionals supported on $\sigma$-Cantor subsets of the leaves or they must be equivalent to the Lebesgue measure on the leaves.
	
	
	\begin{question}
	Does there exists an ergodic measure $\mu$ preserved by some $f_L$ whose conditional measures are supported on a $\sigma$-Cantor subset of the leaves?
	\end{question}

	\item[b)] Let $\varphi:\mathbb R \times M \to M$ be any $C^1$ flow. The foliation $\mathcal F$ given by the orbits of $\varphi$ is a $\varphi_t$-invariant $C^1$-foliation of $M$ for any fixed $t\in \mathbb R$. There is a natural $\mathcal F$-arc length system in this case given by:
	\[l_x(\gamma) := |l|, \quad \text{with} \quad  \varphi(l,\gamma(0))=\gamma(1).\]
	Assume that almost every $x\in M$ is not a periodic point of $\varphi$.
	Given any $\varphi_t$-ergodic invariant measure $\mu$, it follows from \cite[Example 7.4]{Lindenstrauss} that the disintegration of $\mu$ along $\mathcal F$ is either Lebesgue or atomic.
	
	
	\item[c)] Some skew-products also provide interesting examples. For example take $f: \mathbb T^d \times S^1 \to \mathbb T^d \times S^1$ given by 
	\[f(x,y)=(g(x), R_{\alpha}(y)),\]
	where $g:\mathbb T^d \to \mathbb T^d$ is any homeomorphism and $R_{\alpha}:S^1 \to S^1$ denotes a rotation of angle $\alpha$. In this case, the foliation $\mathcal F$ whose leaves are $\{x\}\times S^1$, $x\in \mathbb T^d$, is $f$-invariant and by taking $l_x$ on $\{x\}\times S^1$ to be given by the usual arc length on $S^1$, we conclude that $\{l_x\}$ is a $\mathcal F$-arc length system.
	In this example it is easy to determine the measurable properties of $\mathcal F$ in the sense that, given a Borel $g$-invariant measure $\nu$, the measure $\nu \times \lambda_{S^1}$\footnote{Here we denote by $\lambda_{S^1}$ the standard Lebesgue measure on $S^1$} is $f$-invariant and, a direct application of the Fubbini Theorem shows that, the disintegration of $\mu$ along $\mathcal F$ has the Lebesgue measures $\lambda_{S^1}$ as its conditional measures. \\
	
	\item[d)]
	Another, more interesting case, is provided by recent results of Bonatti-Zhang \cite{BonattiZhang}. A $C^1$ diffeomorphism $f:M\to M$, on a compact Riemannian manifold $M$, is said to be partially hyperbolic if there is a nontrivial splitting 
\[TM=E^s\oplus E^c\oplus E^u\]
such that 
\[Df(x)E^{\tau}(x)=E^{\tau}(f(x)), \; \tau\in \{s,c,u\}\]
and a Riemannian metric for which there are continuous positive functions $\mu,\hat{\mu},\nu,\hat{\nu}, \gamma, \hat{\gamma}$ with
\[\nu(p),\hat{\nu}(p) <1, \quad \text{and} \quad \mu(p)<\nu(p)<\gamma(p)<\hat{\gamma}(p)^{-1}<\hat{\nu}(p)^{-1}<\hat{\mu}(p)^{-1},\]
such that for any vector $v\in T_pM$,
\[\mu(p)||v||<||Df(p)\cdot v|| < \nu(p)||v||, \; \text{if} \; v\in E^s(p)\]
\[\gamma(p)||v|| < ||Df(p)\cdot v|| < \hat{\gamma}(p)^{-1}||v||, \; \text{if} \; v\in E^c(p)\]
\[\hat{\nu}(p)^{-1}||v||< ||Df(p)\cdot v||<\hat{\mu}(p)^{-1}||v||, \; \text{if} \; v\in E^u(p).\]
We say that $f$ has topological neutral center if, for any $\varepsilon>0$, there exists $\delta>0$ for which: given any smooth curve $\gamma:[0,1]\to M$ with $\gamma'(t) \in E^c(\gamma(t))$, $0\leq t \leq 1$, if $\text{length}(\gamma) <\delta$ then $\text{length}(f^n(\gamma))<\varepsilon$, for all $n\in \mathbb Z$.
For partially hyperbolic diffeomorphisms with neutral center, the center distribution $E^c$ integrates to a $f$-invariant foliation $\mathcal F^c$ (see \cite[Corollary 7.6]{HHUSurvey}) called center foliation of $f$.
In \cite{BonattiZhang} the authors proved that if $f:M \to M$ is a $C^1$ partially hyperbolic diffeomorphism with neutral center direction, then $f$ admits a continuous $\mathcal F$-arc length system. To understand the measurable properties of the center foliation preserved by such maps, were one of the motivations of this work. As a consequence of our results, the disintegration of any $f$-invariant ergodic probability measure of such maps falls in three possible cases. When the conditional measures have full support, the second author proves in \cite{Ponce2} the occurrence of an invariance principle. Further, if the measure is smooth, full support of the conditional measures imply the Bernoulli property for $f$. If, moreover, $f$ is locally accessible, then $\mathcal F^c$ is as regular as the map $f$ itself. 
	\end{itemize}
	
	\begin{definition}
		Given $y, z, w\in \mathcal{F}(x)$ we say that $y$ is between $z$ and $w$, if there exists a simple arc $\gamma:[0,1]\to \mathcal{F}(x)$ such that $\gamma(0)=z$, $\gamma(1)=w$, $\gamma(t)=y$ for some $t\in (0,1)$, and \[l_{x}(\gamma)=\min\{l_{x}(\alpha) \, : \, \alpha:[0,1]\to \mathcal{F}(x), \, \alpha(0)=z, \, \alpha(1)=w \}.\]
	\end{definition}

	\begin{definition}\label{defi:Fmetric}
		Let $\{l_x\}$ be a $\mathcal{F}$-arc length system. For $x\in M$ we define a metric $d_x$ on $\mathcal{F}(x)$ by
		\[d_x(y,z):=\min \{l_{x}(\gamma) \, : \, \gamma:[0,1]\to \mathcal{F}(x) \; \text{is simple with} \; \, \gamma(0)=y, \, \gamma(1)=z\}.\]
		We call the family $\{d_x\}_x$ the $\mathcal{F}$-metric system associated to the $\mathcal F$-arc length system $\{l_x\}_x$.
	\end{definition}
   \begin{remark}\label{metricsystem}
       By the definition of $d_x$, it is evident that $d_x$ is an additive metric. Specifically, for any $y, z, w \in \mathcal{F}(x)$ such that $y$ lies between $z$ and $w$, the following holds:  
\[
d_x(z, w) = d_x(z, y) + d_x(y, w).
\]  
Moreover, the family $\{d_x\}$ is invariant under $f$, in the sense that  
\[
d_{f(x)}(f(z), f(y)) = d_x(z, y).
\]

It is not true that $\{d_x\}_{x\in M}$ is continuous in the sense that we may have sequences $x_n \rightarrow x$, $y_n \rightarrow y$, with $y_n \in \mathcal F(x_n)$, $y\in \mathcal F(x)$ but $d_{x_n}(x_n,y_n) \nrightarrow d_{x}(x,y)$. Indeed this happens, for example, for compact foliations where the leaves do not have uniformly bounded length. It is true, however, that this family of metrics are continuous when restricted to plaques inside local charts. We make this property more precise below.
  \end{remark}
 \begin{definition}
Consider $\mathcal F$ a continuous foliation of $M$. A function $F:\bigcup_{x\in M} \mathcal F(x) \times \mathcal F(x) \to [0,\infty)$ will be called plaque-continuous if given any $p\in M$, there exists a local chart $p\in U$ of $\mathcal F$, such that for any sequences  $x_n \rightarrow x$, $y_n \rightarrow y$ with $y_n \in \mathcal F|U(x_n)$, $x\in U$ and $y \in \mathcal F|U(x)$, we have
\[\lim_{n\rightarrow \infty} F(x_n,y_n) = F(x,y).\]
Any such local chart $U$ will be called a continuity-domain of $F$.
 \end{definition}
 
 \begin{definition}
We say that a family of metrics $\{d_x: x\in M\}$, each $d_x$ defined on $\mathcal F(x)$, is plaque-continuous if $F:\bigcup_{x\in M} \mathcal F(x) \times \mathcal F(x) \to [0,\infty)$ defined by 
\[F(x,y):=d_x(x,y),\]
is plaque continuous. In this case if $U$ is a continuity-domain of $F$ we will also say that $U$ is a continuity-domain of $\{d_x\}$.
 \end{definition}
 
 \begin{proposition}
 The metric system given in Definition \ref{defi:Fmetric} is plaque-continuous.
 \end{proposition}
 \begin{proof}
 Let  $\varphi:U \to (0,1)\times V \subset \mathbb R^n$ be a local chart of $\mathcal F$ where $\varphi^{-1}((0,1) \times \{c\})$, $c\in V$, are the plaques of $\mathcal F$ in $U$. For any $p\in U$, consider $\xi: W\subset U \to (0,1)\times K\subset \mathbb R^n$ another local chart centered in $p$ such that for any 
 \[z\in W \Rightarrow l_z(\mathcal F|U(z) ) > 3\cdot l_z(\mathcal F|W(z)).\]
 This can be done by the continuity of $\{l_x\}$.  In particular, for any $x\in W, y\in \mathcal F|W(x) = \xi^{-1}((0,1),c')$, the simple curve $\gamma(t) = \xi^{-1}((1-t)\xi(x)+t\xi(y), c' )$ minimizes the $l_x$-length connecting $x$ and $y$, that is, $d_x(x,y) = l_x(\gamma)$.
 
 On that account, consider $x\in W, y\in \mathcal F|W(x) = \xi^{-1}((0,1),c')$ and sequences $x_n \in W, y_n \in \mathcal F|W(x_n) = \xi^{-1}((0,1),c_n)$ with $x_n \rightarrow x$ and $y_n \rightarrow y$. Let $\gamma$ be defined as in the previous paragraph and $\gamma_n(t):=\xi^{-1}((1-t)\xi(x_n)+t\xi(y_n), c_n )$. By the convergence of the sequences we have $\gamma_n \rightarrow \gamma$. But by the previous discussion on the choice of the local chart $W$ we have
 \[d_{x_n}(x_n,y_n) = l_{x_n}(\gamma_n) \quad \text{and} \quad d_x(x,y) = l_x(\gamma).\]
 We then conclude by the continuity of $\{l_x\}$ that $\lim_{n\rightarrow \infty} d_{x_n}(x_n,y_n) =\lim_{n\rightarrow \infty}  l_{x_n}(\gamma_n)= l_x(\gamma) = d_x(x,y)$.
 
Therefore $\{d_x\}$ is plaque-continuous as we wanted to show.
 \end{proof}

	\begin{lemma}\label{lemma:opencont}
	    Given any local open transversal $T$ to $\mathcal F$, for any $r$ small enough, the set
	    \[S:=\bigcup_{x\in T}B_{d_x}(x,r)\]
	    is open.
	\end{lemma}
	\begin{proof}
	Assume that $T$ is a local transversal associated to a local chart $(U,\varphi)$. Let $r>0$ small enough we have $B_{d_x}(x,r)\subset U$ for every  $x\in T$. In particular $U\setminus \overline{T}$ has two open connected components $U_1$ and $U_2$ with $\overline{U_1}\cap \overline{U_2} = \overline{T}$.
	
	Since $U$ is a local chart, we may consider an orientation on the $\mathcal F|U$-plaques. Assume that $S$ is not open. Then, there exists $y\in S$ and a sequence $y_k \notin S$, with $y_k\rightarrow y$.
	
	Consider $x\in T$ such that $y\in B_{d_x}(x,r)$ and denote by $\varphi$ the flow on the $\mathcal F|U$-plaques induced by the orientation fixed before and such that
	\[d_p(\varphi_t(p), p)=|t|,\]
	whenever $\varphi_t(p)$ is defined.
	Let $t_0$ be such that $x=\varphi_{t_0}(y)$. As $y\in B_{d_x}(x,r)$, there exists $\delta>0$ for which
	\[\varphi_t(y)\in S, \quad t\in [t_0-\delta, t_0+\delta].\]
	Now, by the plaque continuity and the fact that $y_k\rightarrow y$, we have
	\[\varphi_{t_0-\delta}(y_k)\rightarrow \varphi_{t_0-\delta}(y), \quad \varphi_{t_0+\delta}(y_k)\rightarrow \varphi_{t_0+\delta}(y).\]
	Observe that $\varphi_{t_0-\delta}(y)$ and $\varphi_{t_0+\delta}(y)$ belong to different connected components, thus, for $k$ large enought the same happens for $\varphi_{t_0-\delta}(y_k)$ and $\varphi_{t_0+\delta}(y_k)$. Since $\gamma_k:=\{\varphi_{t}(y_k): \quad t\in [t_0-\delta, t_0+\delta]\}$ is an arc with points in the interior and in the exterior of $U_1$, it must intersect its boundary, namely $T$. This implies that $y_k\in S$ for large $k$, yielding a contradition.
	
	That is, $S$ is open as we wanted to show.

	
	\end{proof}

	\begin{proposition}\label{Brcontida}
	Given a finite open cover $\mathcal U$ of $M$ by local charts of $\mathcal F$, there exists $\mathfrak r>0$ such that for all $x\in M$, there is $U\in \mathcal U$ with
	\[B_{d_x}(x,\mathfrak r) \subset U.\]
	\end{proposition}
	\begin{proof}
	    For each $x\in M$, take any $U_x\in \mathcal U$ with $x\in U_x$. There exists $r_x>0$ for which $B_{d_x}(x,r_x) \subset U_x$. By plaque continuity of $\{d_x\}$ there exists a neighborhood $x\in V_x \subset U_x$ for which
	    \[y\in V_x \Rightarrow B_{d_y}(y,r_x) \subset U_x.\]
	    Since $M$ is compact we may cover $M$ with a finite number of neighborhoods $V_{x_i}$, $1\leq i \leq l$. Take $\mathfrak r = \min \{r_{x_i}: 1\leq i \leq l \}$.
	\end{proof}
	
	In the sequel we will prove a technical result which will be used along the proof of the main theorem. Namely, we prove that the continuous translation of a measurable set along the foliation $\mathcal F$ is also a measurable set. 
	
	\begin{lemma}\label{conjuntofluxo} 
	There exists $t_0>0$ such that for every $0\leq t\leq t_0$ and every  Borel subset $A\subset M$ the set 
		\begin{equation}
			\Phi_t(A):=\{x\in M : d_x(x,A)<t\},
		\end{equation}
		is a measurable set.
	\end{lemma}
	\begin{proof} 
	Let $\mathcal U$ be a finite cover of $M$ by local charts which are continuity-domains of $\{d_x\}$. Consider $\mathfrak r$ the number given by Proposition \ref{Brcontida}. In particular the family $\{U_{\mathfrak r/2} : U\in \mathcal U\}$,
	defined by
	\[U_{\mathfrak r/2} = \{x\in U : d_x(x,\partial U)\geq \mathfrak r/2\},\]
	is still a cover of $M$. Let $A\subset M$ be a Borel subset. Observe that
	\[\Phi_{t}(A\cap U_{\mathfrak r/2}) \subset U, \quad U\in \mathcal U, \quad t<\mathfrak r/2.\]
		
		We will prove that for $U\in \mathcal U$, the subset $\Phi_t(A \cap U_{\mathfrak r/2})$ is measurable. 
		
		Let $\varphi_U:U\to B_1^{n-1}(0)\times (0,1)$ be a local chart of $\mathcal F$, and inside $U$ consider the orientation in the plaques $\mathcal{F}(x)$ induced by the orientation in the line segments of the form $\{\overline{x}\}\times (0,1)\subset \mathbb R^{n-1}\times \mathbb R$. This orientation induces, at each plaque, an order relation which we will denote by $\prec$ (the plaque being implicit in the context).
		
		Now for $s\in [-t,t]$, with $0\leq t< \mathfrak r/2$ fixed, we define  $\phi_s^U:U_{\mathfrak r/2}\to U$ by:
		\begin{itemize}
			\item for $s>0$, $\phi^U_s(x)$ is the only point of the plaque $\mathcal{F}|U(x)$ such that $d_x(x,\phi_s^U(x))=s$ and $x\prec \phi_s^U(x)$;
			\item for $s<0$, $\phi_{s}^U(x)$ is the unique point of the plaque $\mathcal{F}|U(x)$ such that $d_x(x,\phi_s^U(x))=-s$ and $\phi_s^U(x)\prec x$.
		\end{itemize}
		
		Observe that $\phi^U_s$ is continuous for every $|s|<t$ since $U$ is a continuity-domain of $\{d_x\}$ and, consequently, it is a homeomorphism. Thus $\phi_s^U(A\cap U_{\mathfrak r/2})$ is a measurable subset of $M$ for every $s\in [-t,t]$. 
		
		Now, for each $1\leq i \leq n$ take  
		\[
		\Phi_t^U(A):=\bigcup_{\substack{q\in\mathbb{Q} \\ q<t}}\phi_q^U(A\cap U_{\mathfrak r/2}), \quad 0\leq t <\mathfrak r/2.
		\]
		Notice that $\Phi_t^U(A)$ is a measurable set, since each set in the countable union is measurable as we have proved before. Consequently,
		\[
		\Phi_t(A)=\bigcup_{U\in \mathcal U}\Phi_t^U(A),
		\]
		is a measurable set, as we wanted to show.
	\end{proof}
	
	\begin{definition}\label{def.lambdax}
		Let $\{l_x\}$ be a $\mathcal{F}$-arc length system. Then, we have a well defined homeomorphism
		\[h_x: \mathcal F(x) \to F,\]
		where $F=\mathbb R$ or $F=S^1$, $h_x(x)=0$\footnote{Here we are using the identification $S^1=[0,1]/\sim$ where $0\sim1$, thus the point $0$ stands for the equivalence class of $0$ in $S^1$.}, and such that, for any simple arc $\gamma:[0,1]\to \mathcal F(x)$ we have
		\[l_x(\gamma[0,1]) = \lambda(h_x(\gamma[0,1])),\]
		where $\lambda$ denotes the Lebesgue measure on $F$. In particular $\lambda(h_x(\gamma[0,1]))$ is the size of the interval $h_x(\gamma[0,1])$.
		We now define the measure $\lambda_x$ on $\mathcal F(x)$ given by:
		\[\lambda_x = (h_x^{-1})_*\lambda.\]
		
	\end{definition}
	Note that if $\gamma[0,1]$ is a simple arc in $\mathcal F(x)$ then,
	\[\lambda_x(\gamma[0,1]) =\lambda(h_x(\gamma[0,1])) = l_x(\gamma[0,1]). \]
	Consequently, the measure $\lambda_{x}$ is a doubling measure\footnote{
		Recall that given a metric space $(X,d)$, a measure $\nu$ on $X$ is said to be a \textbf{doubling measure} if there exists a constant $\Omega>0$ such that for any $x\in X$ and any $r>0$ we have
		\[\nu(B(x,2r)) \leq \Omega \cdot \nu(B(x,r)).\]}.



	\section{Properties of non-atomic disintegrations over a continuous one-dimensional foliation} \label{sec:FiberedSpaces}
	
	The proof of the main result of this paper follows from understanding the topological structure of $\text{supp} \; \mu_x^U$, for a disintegration $\{\mu^U_x\}_{x\in U}$ of $\mu(\cdot |U)$ of $\mu$ on a local chart $U$. To this end we first need to understand the behavior, in terms of measure theory, of the map
	\[(x,r) \mapsto \mu^U_x(B_{d_x}(x,r)),\]
	defined for a certain subset of $U\times \mathbb R$. This is the goal of this Section.
	

	Along the rest of the paper we assume the following:
	\begin{itemize}
		\item $\mathcal F$ is a $f$-invariant one dimensional continuous foliation,
		\item $\mathcal U$ is a finite cover of $M$ by local charts $U$ of $\mathcal F$ such that $\overline{U}$ is still inside a local chart of $\mathcal F$,
		\item each $U\in \mathcal U$ is a continuity-domain of $\{d_x\}$,
		\item for each $U\in \mathcal U$, $\{\mu^U_x\}$ is a disintegration of $\mu(\cdot |U)$ along the plaques $\mathcal F|U$,
		\item the disintegration of $\mu$ along $\mathcal F$ is not atomic, in particular, for each $U\in \mathcal U$ there exists a subset $\mathcal A_U \subset U$ with $\mu(\mathcal A_U)=0$ and for which
		\[x\notin \mathcal A_U \Rightarrow \mu^U_x \; \text{is not atomic},\]
		\item $\mathfrak r>0$ is any constant given by Proposition \ref{Brcontida},
		\item $\{\Omega_x\}_{x\in M}$ a disintegration of $\mu$ along $\mathcal F$ and we denote $\omega_x$ any representative of $\Omega_x$, $x\in M$.
	\end{itemize}
	
	We also fix the following notation: for any subset $X\subset M$, we denote by $\mathcal B_X$ the Borel $\sigma$-algebra of $X$ given by the topology induced by that of $M$. It is important to observe that, by definition, for any $U\in \mathcal U$, the set $\mathcal A_U$ is $\mathcal F|U$-saturated in $U$.

    \begin{lemma} \label{lemma:continuousinleaf}
		For each $0<r<\mathfrak r$, $U\in \mathcal U$, $J\subset M$ Borel and $x\in U \setminus \mathcal A_U$, the map
		\[y \mapsto \mu^U_x(J\cap B_{d_x}(y,r)),\]
		is continuous when restricted to the subset $F_{r} \subset \mathcal F|U(x)$ given by
		\[F_r(x) =\{y \in \mathcal F|U(x) : \; B_{d_x}(y,r) \subset U\}. \]
        In particular, for each $0<r<\mathfrak r$, $U\in \mathcal U$ and $x\in U \setminus \mathcal A_U$, the map
		\[y \mapsto \mu^U_x(B_{d_x}(y,r)),\]
		is continuous when restricted to $F_{r} \subset \mathcal F|U(x)$.
	\end{lemma}
	\begin{proof}
		Take  $x\in U$, $U\in \mathcal U$ and $J\subset M$ as in the statement. Let $y_n \rightarrow y$, $y_n, y \in F_r(x)$. We want to show that for $r>0$
		\[\lim_{n\rightarrow \infty} \mu^U_x(J\cap B_{d_{x}}(y_n,r)) = \mu^U_x(J\cap B_{d_x}(y,r)). \]
		Given any $k\in \mathbb N$, since $\mu^U_x$ is not atomic, we have that 	
		\begin{eqnarray*}
			\mu^U_x(\partial B_{ d_x}(y,r))  =0 \; \text{and} \;  \mu^U_x(\partial B_{d_x}(y_n,r))  =0, \forall n \in \mathbb N,
		\end{eqnarray*}
		where $\partial B_{d_x}$ denotes the boundary of the set inside the leaf $\mathcal F(x)$.
		Now, let $B_n:=(J\cap B_{d_x}(y_n,r)) \Delta (J\cap B_{d_x}(y,r))$ where $Y\Delta Z$ denotes the symmetric difference of the sets $Y$ and $Z$. From standard measure theory:
		\[ \limsup_{n\rightarrow \infty}  \mu^U_x(B_n) \leq \mu^U_x \left( \limsup_{n \rightarrow \infty} B_n \right).\]
		Thus
		\begin{align*}
			\limsup_{n\rightarrow \infty}  \mu^U_x(B_n) & \leq  \mu^U_x \left( \bigcap_{m=1}^{\infty} \bigcup_{n\geq m} B_n \right).\end{align*}
		We know that $z\in \bigcap_{m=1}^{\infty} \bigcup_{n\geq m} B_n $ if, and only if, there exists $n_1<n_2<n_3< \ldots$ such that $z \in B_{n_1}\cap B_{n_2} \cap \ldots$. Therefore, for each $i$ either:
        \begin{itemize}
            \item $z \in (J\cap B_{d_x}(y_{n_i},r)) \setminus (J\cap B_{d_x}(y,r)) \Rightarrow z \in J, z\in B_{d_x}(y_{n_i},r)$ and $z\notin B_{d_x}(y,r)$
            \item $z \in (J\cap B_{d_x}(y,r)) \setminus (J\cap B_{d_x}(y_{n_i},r)) \Rightarrow z \in J, z\in B_{d_x}(y,r)$ and $z\notin B_{d_x}(y_{n_i},r)$
        \end{itemize}
        Observe that $z$ cannot fall in the second case infinitely often since $\bigcap_i (B_{d_x}(y,r)\setminus B_{d_x}(y_{n_i},r)) = \emptyset$. Therefore, passing to a subsequence if necessary we have 
        \[ z \in J, z\in B_{d_x}(y_{n_i},r)\quad \text{and} \quad z\notin B_{d_x}(y,r)\]
        which implies, by taking $i \rightarrow \infty$, that $z\in J\cap \partial B_{d_x}(y,r) \subset \partial B_{d_x}(y,r)$. Consequently,
        \[\limsup_{n\rightarrow \infty}  \mu^U_x(B_n)
		\leq \mu^U_x(\partial B_{d_x}(y,r)) = 0.\]
        Therefore 
		\[\lim_{n\rightarrow \infty} \mu^U_x (J\cap B_{d_x}(y,r) \setminus J\cap B_{d_x}(y_n,r)) = \lim_{n\rightarrow \infty} \mu^U_x (J\cap B_{d_x}(y_n,r) \setminus J\cap B_{d_x}(y,r))= 0\] 
		and consequently 
		\begin{align*}
		    \lim_{n\rightarrow \infty} \mu^U_x (J\cap B_{d_x}(y_n,r)) & =  \lim_{n\rightarrow \infty}( \mu^U_x(J\cap B_{d_x}(y_n,r) \setminus J\cap B_{d_x}(y,r)) +  \mu^U_x (J\cap B_{d_x}(y,r)) \\
            & - \mu^U_x(J\cap B_{d_x}(y,r) \setminus J\cap B_{d_x}(y_n,r)) = \mu^U_x (J \cap B_{d_x}(y,r)),
		\end{align*}
		as we wanted to show.
        The second claim is a consequence of the first by taking $J=M$.
	\end{proof}

    \begin{proposition} \label{prop:unbounded.measurable2}
	Let $U\in \mathcal U$ and $0<r<\mathfrak r$. Consider $(V,\varphi)$ a local chart inside $U$ such that
		\[x\in V \Rightarrow B_{d_x}(x,r) \subset U,\]
        and $J\subset M$ a Borel subset of $M$.
		Then, restricted to $V\setminus \mathcal A_U$ the map
		\[x\mapsto \mu^U_x(J\cap B_{d_x}(x,r)),\]
		is $\mathcal B_{V\setminus \mathcal A_U}$-measurable, and consequently $\mathcal B_{U\setminus \mathcal A_U}$-measurable as $V\subset U$.
        In particular, the map 
        \[x \in V\setminus \mathcal A_U\mapsto \mu^U_x(B_{d_x}(x,r)),\]
		is $\mathcal B_{U\setminus \mathcal A_U}$-measurable.
	\end{proposition}
	\begin{proof}

		If $x$ and $y$ belong to the same $\mathcal F$-plaque in $U$ then $\mu^U_x=\mu^U_y$ . We already know that for all Borel subset $W\subset M$
		\[x\in V \mapsto \mu^U_x(W \cap U),\] is Borel measurable.
		
		Consider the local chart be given by the homeomorphism $\varphi:V \to (0,1)\times G$, where $G\subset \mathbb R^{n-1}$ is an open subset. Setting $g_r:V \to [0,\infty)$ as
		\[g_r(x) = \mu^U_x(J\cap B_{d_x}(x,r)),\]
		we have
        \[g_r \circ \varphi^{-1}(x_1,x_2) = \mu^U_{\varphi^{-1}(x_1,x_2)}(J\cap B_{d_{\varphi^{-1}(x_1,x_2)}}(\varphi^{-1}(x_1,x_2),r)).\]
		Let us prove that this function is continuous in $x_1$ and Borel measurable in $x_2$.
		
		By Lemma \ref{lemma:continuousinleaf},  restricted to $V\setminus \mathcal A_U$ we already have the continuity of $g_r\circ \varphi^{-1}$ the first coordinate, since the second coordinate being fixed means we are evaluating the function on a single plaque where the conditional measure is non-atomic. Now, fix the first coordinate $x_1$ and consider the transversal $T=\{x_1\} \times B_1^{n-1}(0)$. By Lemma \ref{lemma:opencont},
		the set
		\[S:=\bigcup_{x\in \varphi^{-1}(T)}B_{d_x}(x,r),\]
		is an open subset of $M$ and consequently $S\cap J =\bigcup_{x\in \varphi^{-1}(T)}J\cap B_{d_x}(x,r) $ is Borel. Thus, $y \mapsto \mu^U_y(S\cap J)$ is a Borel measurable function in $V$, which implies that $g_r\circ \varphi^{-1}(x_1, \cdot)$ is $\mathcal B_T$-measurable. In particular, its restriction to $T\cap \varphi(V\setminus \mathcal A_U) = T\setminus \varphi(\mathcal A_U)$ is a $\mathcal B_{T\setminus \varphi(\mathcal A_U)}$-measurable map. But observe that for $x\in T$ we have
		\[\mu^U_x(S\cap J) = \mu^U_x(J\cap B_{d_x}(x,r)).\]
		Therefore, for $x_1$ fixed the map $x_2 \in G \setminus \pi_2(\varphi(\mathcal A_U)) \mapsto \mu^U_{\varphi^{-1}(x_1,x_2)}(J\cap B_{d_{\varphi^{-1}(x_1,x_2)}}(\varphi^{-1}(x_1,x_2),r))$ is $\mathcal B_{G \setminus \pi_2(\varphi(\mathcal A_U))}$-measurable, where $\pi_2:(0,1)\times G \mapsto G$ is the projection onto the second coordinate.
		
		Consequently, $g_r\circ \varphi^{-1}$ restricted to $((0,1)\times G)\setminus \varphi(\mathcal A_U)$ is a jointly measurable function with respect to the product $\sigma$-algebra $\mathcal B_{(0,1)} \times \mathcal B_{G \setminus \pi_2(\varphi(\mathcal A_U))}$ (see for example \cite[Lemma 4.51]{InfDimAna}). As $\varphi$ is a homeomorphism, we conclude that $g_r$ is $\mathcal B_{V\setminus \mathcal A_U}$-measurable, as we wanted to show.		
	\end{proof}
		
	In the following Lemma, we prove that the subset of $M$ consisting of all points $x\in M$ for which there is a ball in $\mathcal{F}(x)$ with null $\mu^U_x$ measure, is a relatively Borel set.
	
	\begin{lemma}\label{lemma:defiZ}
		For each $U\in \mathcal U$, the set
		\[\mathcal Z_U = \bigcup_{x\in U\setminus \mathcal A_U} \mathcal F|U(x) \setminus \text{supp} \; \mu^U_x,\]
		
is $\mathcal B_{U\setminus \mathcal A_U}$-measurable set, meaning $\mathcal{Z}_U=V\cap (U\smallsetminus\mathcal{A}_U)$ for some open set $V$.
	
	\end{lemma}
	\begin{proof}
		First let us give a better formulation for the definition of $\mathcal Z_U$. Observe that 
		\[\mathcal Z_U =\{x \in U \setminus \mathcal A_U : \mu^U_x(I)=0 \; \text{for some open ball} \; x\in I\subset \mathcal F(x)\}.\]
		Consider $\{q_1,q_2,\ldots \}$ an enumeration of $\mathbb Q \cap [0,1]$ and let $\mathcal U$ be the given finite family of local charts covering $M$. For each $U \in \mathcal U$ and $i\in \mathbb N$, define $\phi_i^{U}:U_i\setminus \mathcal A_U\to \mathbb{R}$ by \[\phi_i^{U}(x)=\mu^U_{x}(B_{d_x}(x,q_i)),\]
		where $U_i=\{x\in U:B_{d_x}(x,q_i) \subset U\} $. Observe that we may cover $U_i$ with a countable number of local charts $V_i^j\subset U_i$, $j\in \mathbb N$ and, by Proposition \ref{prop:unbounded.measurable2}, $\phi_i^{U} | V_i^j$ is $\mathcal B_{V_i^j\setminus \mathcal A_U}$-measurable for every $j$. In particular $\phi_i^U$ is $\mathcal B_{U_i\setminus \mathcal A_U}$-measurable for every $i$.
On that account we have that $\mathcal{Z}_i^{U}:=(\phi_i^{U})^{-1}(\{0\})\subset M$ is a $\mathcal B_{U_i\setminus \mathcal A_U}$-measurable subset, in particular, a $\mathcal B_{U\setminus \mathcal A_U}$-measurable subset.
		It is not difficult to see that
		\begin{equation}\label{conjunto0}
			\mathcal{Z}_U=\bigcup_{i=1}^\infty \mathcal{Z}_i^{U}.
		\end{equation} 
		Therefore $\mathcal{Z}_U$ is a $\mathcal B_{U\setminus \mathcal A_U}$-measurable subset as we wanted to show. Moreover, $\mu(\mathcal{Z}_U)=0$.
	\end{proof}

    Consider $\mathcal P$ the set given by
    \[\mathcal P = \{x: \; \exists U,V \in \mathcal U, \; x\in U\cap V, \; \mu^U_x (\cdot | U\cap V) \nsim \mu^V_x (\cdot | U\cap V)\},\]
    that is, $\mathcal P$ is the set of points $x$ for which there exists two local charts $U$ and $V$ in $\mathcal U$, both containing $x$, where the respective conditional measures at the plaque of $x$, $\mu^U_x$ and $\mu^V_x$, are not equivalent on the intersection $\mathcal F|U(x) \cap \mathcal F|V(x)$. In particular this set has zero measure by Proposition \ref{prop:disintegration.unbounded}.
	Set $\widetilde{M}:=M\setminus \left(\bigcup_{U\in \mathcal U} (\mathcal Z_U \cup \mathcal A_U) \cup \mathcal P\right)$. 
	Let 
	\begin{equation}\label{defiM0}
    M_0:= \bigcap_{n\in \mathbb Z}f^n(\widetilde{M}).\end{equation}
	As $\mu(\widetilde{M})=1$, we have $\mu(M_0)=1$.
	For each $x\in M_0$, we will denote by $\mu_x$ the measure on $B_{d_x}(x,\mathfrak r)$ given by the conditional measure $\mu^U_x$, for some $U\in \mathcal U$ with $x\in U$, normalized to give weight exactly one to $B_{d_x}(x,\mathfrak r)$, 
	that is, for a measurable $F\subset \mathcal F|U(x)$ 
	\begin{equation}\label{eq:defimux}\mu_x(F) = \mu^U_x(F | B_{d_x}(x,\mathfrak r)).\end{equation}
	
	
	Given any $y\in B_{d_x}(x,\mathfrak r)\cap M_0$, the measures $\mu_y$ and $\mu_x$ are proportional to each other at the intersection $B_{d_x}(x,\mathfrak r) \cap B_{d_y}(y,\mathfrak r)$ by Proposition \ref{prop:disintegration.unbounded}, that is, there exists a constant $\beta$ for which $\mu_y = \beta \cdot \mu_x$ restricted to $B_{d_x}(x,\mathfrak r) \cap B_{d_y}(y,\mathfrak r)$.

	
	In particular, evaluating both sides at $B_{d_x}(x,\mathfrak r) \cap B_{d_x}(y,\mathfrak r)$ we see that
	\[\beta \cdot \mu_x(B_{d_x}(x,\mathfrak r) \cap B_{d_x}(y,\mathfrak r)) = \mu_y(B_{d_x}(x,\mathfrak r) \cap B_{d_x}(y,\mathfrak r))\]
	\[\Rightarrow \beta=\frac{\mu_y(B_{d_x}(x,\mathfrak r) \cap B_{d_x}(y,\mathfrak r))}{\mu_x(B_{d_x}(x,\mathfrak r) \cap B_{d_x}(y,\mathfrak r))}.\]

    \begin{corollary} \label{cor:continuanafolha}
		For
		each $0<r<\mathfrak r$, $x\in M_0$ and $J\subset M$ Borel, the map
		\[
		y \in B_{d_x}(x,\mathfrak r)\cap M_0 \mapsto \mu_y(J \cap B_{d_x}(y,r)) ,
		\]
		is continuous.
	\end{corollary}
	\begin{proof} For a certain $0<r<\mathfrak r$ fixed, take any $x \in M_0$. Let $y\in B_{d_x}(x,\mathfrak r) \cap M_0$ and $U \in \mathcal U$ with $B_{d_x}(y,\mathfrak r) \subset U$, take $y_n \in B_{d_x}(x,\mathfrak r)\cap M_0$ with $y_n \rightarrow y$ as $n\rightarrow \infty$ and $B_{d_x}(y_n,\mathfrak r) \subset U$. 
		
		By definition,
		\[\mu_y = \mu_y^U(\cdot |B_{d_x}(y,\mathfrak r)), \; \mu_{y_n} = \mu_{y_n}^U(\cdot |B_{d_x}(y_n,\mathfrak r)). \]
		Therefore for $n\in \mathbb N$ such that $B_{d_x}(y,r)\subset B_{d_x}(y_n,\mathfrak r)$ and $B_{d_x}(y_n,r)\subset B_{d_x}(y,\mathfrak r)$, 
		\begin{equation}\label{eq:mexe}
		\mu_{y_n}(J\cap B_{d_x}(y_n,r)) = \frac{\mu_{y_n}^U(J\cap B_{d_x}(y_n,r))}{\mu_{y_n}^U(B_{d_x}(y_n,\mathfrak r))} = \frac{\mu_y^U(J\cap B_{d_x}(y_n,r))}{\mu_{y}^U(B_{d_x}(y_n,\mathfrak r))}.\end{equation}
		By Lemma \ref{lemma:continuousinleaf} we have $\mu_y^U(J\cap B_{d_x}(y_n,r))\rightarrow \mu_y^U(J\cap B_{d_x}(y,r))$ and $\mu_{y}^U( B_{d_x}(y_n,\mathfrak r)) \rightarrow \mu_{y}^U(B_{d_x}(y,\mathfrak r))$ as $n\rightarrow \infty$. Therefore $\mu_{y_n}(J\cap B_{d_x}(y_n,r))\rightarrow \mu_{y}(J\cap B_{d_x}(y,r))$ as we wanted to show.
	\end{proof}

\begin{corollary} \label{cor:jointlymeasurable}
	  For each $x\in M_0$, and $J\subset M$ Borel, the map
		\[r \in [0,\mathfrak r] \mapsto \mu_x(J\cap B_{d_x}(x,r)),\]
		is continuous. 
		Furthermore the map
		\begin{equation}\label{eq:primafunc1}
        (x,r) \in M_0 \times [0,\mathfrak r] \mapsto \mu_x(J\cap B_{d_x}(x,r)),\end{equation}
		is jointly measurable. In particular
        \begin{equation}\label{eq:primafunc}(x,r) \in M_0 \times [0,\mathfrak r] \mapsto \mu_x(B_{d_x}(x,r)),\end{equation}
        is jointly measurable.
	\end{corollary}
	\begin{proof}
		Let $x\in M_0$, first, let us prove that $r \in [0,\mathfrak r] \mapsto \mu_x(J\cap B_{d_x}(x,r))$ is a continuous function. Let $0=r<\mathfrak r$ and $r_n \in [0,\mathfrak r] \searrow r$ (if $r=\mathfrak r$ the argument is analogous), hence  \[\mu_{x}(J\cap B_{d_x}(x,r_n)) = \mu_x(J\cap B_{d_x}(x,r)) + \mu_x(J\cap B_{d_x}(x,r_n)\setminus J\cap B_{d_x}(x,r)).\] As $\mu_x$ is non-atomic we have 
		\[\lim_{n\rightarrow \infty}\mu_x(B_{d_x}(x,r_n)\setminus B_{d_x}(x,r)) =0.\]
		Then,
        \[\mu_x(J\cap B_{d_x}(x,r_n)\setminus J\cap B_{d_x}(x,r)) =\mu_x(J \cap ( B_{d_x}(x,r_n)\setminus  B_{d_x}(x,r))) \rightarrow 0,\]
		showing the first part of the statement.
		
		Let us show the second statement. For each $x\in M_0$, let $x\in V_x$ be a local chart with
		\[y\in V_x \Rightarrow B_{d_y}(y,\mathfrak r) \subset U_x, \quad \text{for some} \; U_x \in \mathcal U.\]
		As $M$ is compact we may cover $M$ with a finite number of such local charts, say $V_1,V_2, \ldots, V_l$ and call $U_1, U_2, \ldots, U_l$ the associated local charts in $\mathcal U$.
		For any $j$, consider
		\[y\in V_j \mapsto \mu_y(J\cap B_{d_y}(y,r)).\]
		Observe that 
		\[\mu_y(J\cap B_{d_y}(y,r)) =\frac{\mu^{U_j}_y(J\cap B_{d_y}(y,r))}{\mu^{U_j}_y(B_{d_y}(y,\mathfrak r))}. \]
		Therefore, by Proposition \ref{prop:unbounded.measurable2}, $y\in V_j \cap M_0 \mapsto \mu_y(J\cap B_{d_y}(y,r))$ is a $\mathcal B_{U_j\setminus \mathcal A_{U_j}}$-measurable map. As $j$ is arbitrary, $y\in M_0 \mapsto \mu_y(J\cap B_{d_x}(y,r))$ is a $\mathcal B_{U\cap M_0}$-measurable map.
		
		Thus, the map given by \eqref{eq:primafunc1} is jointly measurable as it is continuous in the first coordinate and $\mathcal B_{M_0}$-measurable in the second. The second conclusion follows from taking $J=M$.
	\end{proof}
		
	\section{Proof Theorem \ref{theorem:A}} \label{sec:DDL}


	First of all, we can assume that $(M,\mu)$ is an atom-less probability space and that the disintegration of $\mu$ along $\mathcal F$ is not atomic, otherwise there would be nothing to do. We also consider the same objects fixed in the beginning of Section \ref{sec:FiberedSpaces}.
	
	
	Now we define the distortion of the disintegration  measure of $\mu$ relative to the metric system in each leaf $\mathcal{F}(x)$ by the following.
	\begin{definition} \label{defi:distortion2}
		Let $(M,\mu)$ be a probability space and $\mathcal F$ be a continuous one-dimensional foliation of $M$. Let $\{\mu_x\}$ be the system of conditional measures along $\mathcal F$ given by \eqref{eq:defimux}, for $x\in M_0$\footnote{see Lemma \ref{lemma:defiZ} and recall that $\mu(M_0)=1$.}, and let $d=\{d_x\}$ be the $\mathcal F$-metric system induced by the arc-length system $\{l_x\}$ as in Definition \ref{defi:Fmetric}. We define the $\mu$-distortion of the $\mathcal{F}$-metric system by
		\[\Delta(x)= \begin{cases}
			\limsup_{\varepsilon \rightarrow 0}\frac{\mu_x(B_{d_x}(x,\varepsilon))}{2\varepsilon} := \lim_{n\rightarrow \infty}\left( \sup_{\varepsilon \leq 1/n} \frac{\mu_x(B_{d_x}(x,\varepsilon))}{2\varepsilon}\right) &\text{ if }  \quad x\in  M_0\\
			0 &\text{ if } \quad x\notin M_0.
		\end{cases} \]
	\end{definition}

	Recall that $B_{d_x}(x,\varepsilon)$ is the ball inside $\mathcal F(x)$, centered in the point $x$ and with radius $\varepsilon$ with respect to the metric $d_x$.
\begin{lemma}
The map $x \in M\mapsto \Delta(x)$ is measurable.
\end{lemma}
\begin{proof}
Clearly it is sufficient to check the measurability of $\Delta(x)$ in $M_0$.
    By Corollary \ref{cor:jointlymeasurable}, for any fixed $x\in M_0$ the function $w(x,\cdot):(0,\mathfrak r) \rightarrow \mathbb (0,\infty)$ is continuous. Thus, for $n$ small enough, 
    \[\sup_{\varepsilon \leq 1/n} \frac{\mu_x(B_{d_x}(x,\varepsilon))}{2\varepsilon} = \sup_{r\in \mathbb Q \cap (0,1/n] } \frac{\mu_x(B_{d_x}(x,r))}{2r}, \quad \forall x\in M_0.\]
For each $n\in \mathbb N$, the map $L_n:M_0 \to \mathbb R$, $x\in M_0 \mapsto \sup_{r\in \mathbb Q \cap (0,1/n] } \frac{\mu_x(B_{d_x}(x,r))}{2r}$ is measurable, as it is defined by the supremum over a countable number of measurable maps. Consequently, since $\Delta(x)= \lim_{n\rightarrow \infty} L_n(x)$ for all $x\in M_0$, we conclude that $\Delta(x)$ is measurable as we wanted.
\end{proof}

	Observe that, a priori, $\Delta(x)$ is a  measurable function but it is not immediately true that $\Delta(x)<\infty$ for $\mu$-almost every $x$. Also note that, 
	\[f_{*}\mu_x = \mu_{f(x)} \;\text{ and }\;f(B_{d_x}(x,\varepsilon)) = B_{d_{f(x)}}(f(x),\varepsilon),\]
	since
	\[d_{f(x)}(f(x),f(y))=d_x(x,y),\] 
	we conclude that $\Delta(x)$ is $f$-invariant map.
	By ergodicity of $f$ it follows that $\Delta(x)$ is constant almost everywhere, let us call that constant by $\Delta$, that is for almost every $x$:
	\begin{equation}\label{eq:delta}
		\Delta(x)  = \Delta.
	\end{equation}
	
	Let $\widehat{M}\subset M_0$ be a Borel $f$-invariant full measure set of points $x$ for which \eqref{eq:delta} occurs.

	\subsection{Technical Lemmas for the case $\Delta = \infty$} \label{sec:infinitycase}
	
	\begin{lemma} \label{lema:sequenciaboa2} 
		If $\Delta=\infty$, there exists a sequence $\varepsilon_k\rightarrow 0$, as $k\rightarrow +\infty$, and a full measure subset $R^{\infty} \subset \widehat{M}$ such that
		\begin{itemize}
			\item[i)] $R^{\infty}$ is $f$-invariant;
			\item[ii)]for all $x \in R^{\infty}$ we have 
			\begin{equation}\label{eq:infity}
				\frac{\mu_x(B_{d_x}(x,\varepsilon_k))}{2\varepsilon_{k}} \geq k .\end{equation}
		\end{itemize} 
	\end{lemma}
	\begin{proof}
		Let $k\in \mathbb N^{*}$ arbitrary. Since $\Delta(x) = \Delta$ for every $x \in \widehat{M}$, define
		\[\varepsilon_k(x):= \sup \left\{\varepsilon \leq 1: \frac{ \mu_x(B_{d_x}(x,\varepsilon))}{ 2\varepsilon} \geq k \right\}, \quad x\in \widetilde{M}.\]
		
		\noindent {\bf Claim:},
		The function $\varepsilon_k(x)$ is measurable for all $k \in \mathbb N$.
		\begin{proof}
			Define 
			\[
			w(x,\varepsilon) =  \frac{ \mu_x(B_{d_x}(x,\varepsilon))}{2\varepsilon}. 
			\]
			
			By Corollary \ref{cor:jointlymeasurable}, for any $x\in M_0$ the function $w(x,\cdot):(0,\mathfrak r) \rightarrow \mathbb (0,\infty)$ is continuous and, for $0<\varepsilon<\mathfrak r$ fixed the function $w(\cdot, \varepsilon):M_0 \rightarrow \mathbb (0,\infty)$ is measurable function by Proposition \ref{prop:unbounded.measurable2}.
			Given any $k\in \mathbb N$, $\beta>0$, the continuity of $w(x,\cdot)$ implies that 
			\begin{align*}
				\varepsilon_k^{-1}((0,\beta)) = & \{x: \varepsilon_{k}(x) \in (0,\beta)\} \\
				= & \bigcap_{\beta\leq r \leq 1} w(\cdot, r)^{-1}([0, k)) \\
				= & \bigcap_{\beta\leq r \leq 1, r\in \mathbb Q} w(\cdot, r)^{-1}([0, k)). \end{align*}
			Therefore $\varepsilon_k^{-1}((0,\beta))$ is measurable, as it is a countable intersection of measurable subsets of $M_0$, and consequently $\varepsilon_k$ is a measurable function for every $k$.
		\end{proof}
		
		Note that $\varepsilon_k(x)$ is $f$-invariant. Thus, by ergodicity, for every $k\in \mathbb{N}$ the function $\varepsilon_k$ is constant almost everywhere, let $R^{\infty}_k$ be a full measure set such that $\varepsilon_k(x)$ is constant equal to $\varepsilon_k$.
		It is easy to see that the sequence $\varepsilon_k$ goes to $0$ as $k$ goes to infinity. Take
		$\widetilde{R}^{\infty}:= \bigcap_{k=1}^{+\infty} R^{\infty}_k$.
		Since each $R^{\infty}_k$ has full measure, $\widetilde{R}^{\infty}$ has full measure and clearly satisfies what we want for the sequence $\{\varepsilon_k\}_{k}$.
		Finally, take $R^{\infty}=\bigcap_{i\in\mathbb{Z}} f^{i}(\widetilde{R}^{\infty})$. The set $R^{\infty}$ is $f$-invariant, has full measure and satisfies $(i)$ and $(ii)$. 
	\end{proof}
	
	We now set, for each $U\in \mathcal U$, $x\in U\setminus (\mathcal Z_U \cup \mathcal A_U)$,
	\[\Pi^{\infty}_{x,U}:= \left\{y \in \mathcal F|U(x) \setminus \mathcal Z_U : \; \frac{1}{2\varepsilon_k}\cdot \frac{ \mu^U_y(B_{d_x}(y,\varepsilon_k))}{\mu^U_y(B_{d_x}(y,\mathfrak r)) } \geq k , \forall k \; \text{with} \; B_{d_x}(y,\varepsilon_k) \subset U \right\},\]
	and
	\[\Pi^{\infty}_U := \bigcup_{x\in U\setminus (\mathcal Z_U \cup \mathcal A_U)} \Pi^{\infty}_{x,U}.\]
	
	Observe that if $x\in R^{\infty}$ then $x\in \Pi^{\infty}_{x,U}$ therefore $R^{\infty}\cap U \subset \Pi^{\infty}_U$. In particular $U\setminus \Pi^{\infty}_U \subset U\setminus R^{\infty}$. Since $\mu(R^{\infty})=1$ then $\Pi^{\infty}_U$ is measurable. Also we can clearly assume that $\varepsilon_k$ is strictly decreasing and $\varepsilon_1<\mathfrak r$.
	
	\begin{lemma} \label{lemma:closed}
		For every $x \in R^{\infty} \cap U$, consider $\delta=\delta(x)>0$ for which
		\[B_{d_x}[x,2\cdot \delta+\mathfrak r]\subset U.\]
		The set $\Pi^{\infty}_{x,U
	} \cap B_{d_x}[x,\delta]$ is a closed subset on the plaque $\mathcal F|U(x)$.
	\end{lemma}
	\begin{proof}
		Let $y_n \rightarrow y$, $y_n \in \Pi^{\infty}_{x,U}\cap B_{d_x}[x,\delta]$, $y\in \mathcal F|U(x)$.
		In particular,
		$B_{d_x}(y_n,\mathfrak r) \subset B_{d_x}(x,\delta+\mathfrak r) \subset U$ and by taking the limit over $n$ we also have
		$B_{d_x}(y,\mathfrak r) \subset B_{d_x}(x,\delta+\mathfrak r) \subset U$. Furthermore, it is clear that $y\in B_{d_x}[x,\delta]$ since this is a closed set.
		By Lemma \ref{lemma:continuousinleaf}, for each $k\in \mathbb N$ the map
		\[y \in B_{d_x}[x,\delta] \subset F_{\varepsilon_k}(x) \mapsto \mu^U_y(B_{d_x}(y,\varepsilon_k)),\]
		is continuous and the same holds for
		\[y \in B_{d_x}[x,\delta] \subset F_{\mathfrak r}(x) \mapsto \mu^U_y(B_{d_x}(y,\mathfrak r)).\]
		 Thus,
		\[\lim_{n\rightarrow \infty} \frac{\mu^U_{y_n}(B_{d_x}(y_n,\varepsilon_k))}{\mu^U_{y_n}(B_{d_x}(y_n,\mathfrak r))} = \frac{\mu^U_{y}(B_{d_x}(y,\varepsilon_k))}{\mu^U_{y}(B_{d_x}(y,\mathfrak r))}, \quad k\geq 1.\]
		which implies that for all $k\geq 1$ we have
		\[\frac{\mu^U_y(B_{d_x}(y,\varepsilon_k))}{2\varepsilon_k \cdot \mu^U_y(B_{d_x}(y,\mathfrak r))} = \lim_{n\rightarrow \infty} \frac{\mu^U_{y_n}(B_{d_x}(y_n,\varepsilon_k))}{2\varepsilon_k \cdot \mu^U_{y_n}(B_{d_x}(y_n,\mathfrak r))} \geq k,\]
		that is, $y\in \Pi^{\infty}_{x,U}$ as we wanted.
		
		%
		%
		%
	\end{proof}

	Now we consider the following sets 
	\[D^{\infty}_U := \mathcal F|U(\Pi^{\infty}_U) \setminus  (\mathcal F|U)(\mathcal{Z_U}).\]
	
	
	We claim that $D^{\infty}_U$ is a $\mathcal A$-measurable subset. In fact, consider the natural projection
	$\pi:U\to U/\mathcal{F}$, as $U$ is an open subset of a manifold (in particular it is a Polish space) we have that $U/\mathcal{F}$ is a Polish space with the quotient topology.  By Lemma \ref{lemma:defiZ} $\mathcal{Z}_U = V \cap (U\setminus \mathcal A_U)$, for some Borel set $V$ . Since $U\setminus \mathcal A_U$ is $\mathcal F|U$-saturated we have
 \[\mathcal{F}|(\mathcal{Z}_U)= \pi^{-1}(\pi(V))\cap (U\setminus \mathcal A_U).\] 
	Note that $\pi(V)$ is a Souslin set \footnote{A subset of a Polish space $Y$ is called a Souslin set, or an analytical set, if it is the image of a Polish space $X$ by a continuous map from $X$ to $Y$.}
 by \cite[Corollary 1.10.9]{bogachevI}. Consequently, $ \pi^{-1}(\pi(V)) $ is measurable in the $\sigma$-algebra $\mathcal{A}$, which we recall is the completion of the Borel sigma algebra $\mathcal B$ , implying that $\mathcal{F}|(\mathcal{Z}_U)$ is $\mathcal A$-measurable.
 
	
	Since $R^{\infty}\cap U\subset \mathcal{F}|U(\Pi^{\infty}_U)$ and $\mu(R^{\infty})=1$ we have that $\mathcal{F}|U(\Pi^{\infty}_U)$ is a $\mathcal A$-measurable subset of $U$, this implies that $D^\infty_U=\mathcal F|U(\Pi^{\infty}_U) \setminus \mathcal F|U(\mathcal{Z}_U)$ is a $\mathcal A$-measurable set as we wanted to show.
	
	Since $D^\infty_U$ is a $\mathcal A$-measurable, by ergodicity of $f$ the $f$-invariant set:
	\[D^{\infty}:=\bigcup_{n\in \mathbb Z, U\in \mathcal U}f^n\left( \bigcup D^{\infty}_U \right),\] must satisfy either $\mu(D^{\infty})=0$ or $\mu(D^{\infty})=1$.
	
	%
	

	\subsection{Technical Lemmas for the case $\Delta < \infty$}
	
	\begin{lemma} \label{lema:sequenciaboa} 
		If $\Delta <\infty$, there exists a sequence $\varepsilon_k\rightarrow 0$, as $k\rightarrow +\infty$, and a full measure subset $R \subset \widehat{M}$ such that
		\begin{itemize}
			\item[i)] $R$ is $f$-invariant;
			\item[ii)]for every $x \in R$, then
			\begin{equation}\label{eq:uniform}
				\left| \frac{ \mu_x(B_{d_x}(x,\varepsilon_k))}{2\varepsilon_{k}} - \Delta  \right|   \leq \frac{1}{k};\end{equation}
		\end{itemize}
	\end{lemma}
	\begin{proof}
		The proof is similar to the proof of Lemma \ref{lema:sequenciaboa2}.
		Let $k\in \mathbb N^{*}$ arbitrary. Since $\Delta(x) = \Delta$, for every $x \in \widehat{M}$  we define
		\[\varepsilon_k(x):= \sup \left\{\varepsilon \leq \frac{1}{k}: \left| \frac{ \mu_x(B_{d_x}(x,\varepsilon))}{2\varepsilon} - \Delta  \right|  \leq \frac{1}{k} \right\}.\]
		Observe that for $x\in \widehat{M}$ such $\varepsilon_k(x)$ exists by the definition of $\Delta$, $\varepsilon_l(x) \rightarrow 0$ as $l\rightarrow \infty$   and
        \[\lim_{l\rightarrow \infty}\left|\frac{ \mu_x(B_{d_x}(x,\varepsilon_l(x)))}{2\varepsilon_l(x)} - \Delta  \right|=0. \]
        
		
		\noindent {\bf Claim:}
		The function $\varepsilon_k(x)$ is measurable for all $k \in \mathbb N$.
		\begin{proof}
			Define \[w(x,\varepsilon) =  \frac{ \mu_x(B_{d_x}(x,\varepsilon))}{ 2\varepsilon}. \]
			As observed in the proof of Lemma \ref{lema:sequenciaboa2}, $r\mapsto w(x,r)$ is continuous and $x\mapsto w(x, r)$ is measurable.
            Given any $k\in \mathbb N$, $k>0$, by the definition of $\varepsilon_k$, for all $\beta> 1/k$ we have $\varepsilon_k^{-1}((0,\beta)) = \widehat{M}$.
			Also, the continuity of $w(x,\cdot)$ implies that for every $\beta\leq 1/k$,
			\begin{align*}
				\varepsilon_k^{-1}((0,\beta)) = & \{x: \varepsilon_{k}(x) \in (0,\beta)\} \\
				= & \bigcap_{\beta\leq r \leq 1/k} w(\cdot, r)^{-1}\left(\left[\Delta+\frac{1}{k}, \infty\right)\right) \cup  w(\cdot, r)^{-1}\left(\left[0, \Delta-\frac{1}{k} \right)\right) \\
				= & \bigcap_{\beta\leq r \leq 1/k, r\in \mathbb Q} w(\cdot, r)^{-1}\left(\left[\Delta+\frac{1}{k}, \infty \right)\right) \cup  w(\cdot, r)^{-1}\left(\left[0, \Delta-\frac{1}{k} \right)\right) .
 			\end{align*}
			Therefore $\varepsilon_k^{-1}((0,\beta))$ is measurable, as it is a countable intersection of measurable sets, and consequently $\varepsilon_k$ is a measurable function for every $k$.
		\end{proof}
		
		As $\varepsilon_k(x)$ is $f$-invariant, by ergodicity we may take the full measure set $R_k$ where $\varepsilon_k(x)$ is constant equal to $\varepsilon_k$.
		The sequence $\varepsilon_k$ goes to $0$ as $k$ goes to infinity, so we set
		$\tilde{R}:= \bigcap_{k=1}^{+\infty} R_k$.
		Since each $R_k$ has full measure, $\tilde{R}$ has full measure and clearly satisfies what we want for the sequence $\{\varepsilon_k\}_{k}$.
		The set $R=\bigcap_{i\in\mathbb{Z}}f^{i}(\widetilde{R})$ is $f$-invariant, has full measure and satisfies $(i)$ and $(ii)$ as we wanted. 
	\end{proof}
	
	Similar to the definitions made in section \ref{sec:infinitycase} we set
	\[\Pi_U := \bigcup_{x\in U\setminus (\mathcal Z_U \cup \mathcal A_U)} \Pi_{x,U}.\] 
	where
	\[\Pi_{x,U}:= \left\{y \in \mathcal F|U(x) \setminus \mathcal Z_U : \; \left|\frac{1}{2\varepsilon_k}\cdot \frac{ \mu^U_y(B_{d_x}(y,\varepsilon_k))}{\mu^U_y(B_{d_x}(y,\mathfrak r)) } - \Delta \right| \leq \frac{1}{k} , \forall k \; \text{with} \; B_{d_x}(y,\mathfrak r) \subset U \right\}.\]
	
	\begin{lemma} \label{lemma:finiteok}
			For every $x \in R \cap U$, consider $\delta(x)>0$ for which
		\[B_{d_x}[x,2\cdot \delta+\mathfrak r]\subset U.\]
		The set $\Pi_{x,U} \cap B_{d_x}[x,\delta]$ is a closed subset on the plaque $\mathcal F|U(x)$.
	\end{lemma}
	\begin{proof}
		Identical to the proof of Lemma \ref{lemma:closed}.
	\end{proof}

	Similar to the definition made in section \ref{sec:infinitycase}, we consider the set	
	\[D_U := \mathcal F|U(\Pi_U) \setminus (\mathcal F|U)(\mathcal{Z_U}),\]
	and
	\[D:=\bigcup_{n\in \mathbb Z, U\in \mathcal U}f^n\left( \bigcup D_U \right).\]
	Similarly $D_U$ is $\mathcal A$-measurable for all $U\in \mathcal U$ and again by ergodicity we have $\mu(D)=0$ or $\mu(D)=1$.

	\vspace{0.3cm}
After proving the auxiliary lemmas for $\Delta^\infty$ (resp. $\Delta$) and obtaining the sets $D$ (resp. $D^\infty$) we divide the next part of the proof into four cases.

	\subsection{Case 1: $\Delta<\infty$ and $\mu(D)=0$.}  \label{sec:Case1} \quad \\
	In this case we will show that the support of $\mu_x$ is a Cantor subset for almost every $x\in M$.
	
	As fixed in the beginning , consider $\{\omega_x\}_x$ the disintegration of $\mu$ along $\mathcal F$ and consider $\mathcal G$ a full measure $\mathcal F$-saturated set of points where 
    \[f^j_*\Omega_x = \Omega_{f^j(x)}, \quad \forall j\in \mathbb Z.\]
 
	Let $\mathcal G^U:=\{x\in U\setminus D: \mu^U_x = \omega_x(\cdot | \mathcal F|U(x))\}\cap \{x: \mu^U_x \; \text{is non-atomic}\}$.
	
	
	Consider:
	\begin{itemize}
	\item $\Phi^U_{1/n}(\mathcal Z_U):=\{x\in U: d_x(x,\mathcal Z_U) <1/n\},$
	    \item
	    $\mathcal E_n = \bigcup_j \mathcal G\cap f^j(\Phi_{1/n}^U(\mathcal Z_U) \cap \mathcal G^U)$.
	    \end{itemize}
	As $\mathcal Z_U = \chi_U \cap (U\setminus \mathcal A_U)$ we have
	\[\Phi^U_{1/n}(\mathcal Z_U) = \{x \in M: d_x(x,\mathcal Z_U) < 1/n\} \cap (U\setminus \mathcal A_U), \]
	which is measurable for $n \geq \hat{n}$, for some $\hat{n}\in \mathbb N$ not depending on $U$, by Lemma \ref{conjuntofluxo} since $\chi_U$ is a Borel subset of $U$. Therefore the set of the second item is a $f$-invariant measurable subset of $M$, thus it either has full or null measure.
	
	Now, again we separate two cases:
	
	\begin{itemize}
	\item Case 3.1: Assume that for all $n\geq \hat{n}$, $\mu(\mathcal E_n)=1$. Then,
	    \[\mathcal E_n^U=\mathcal G^U \cap \left(\bigcup_j \mathcal G\cap f^j(\Phi_{1/n}^U(\mathcal Z_U) \cap \mathcal G^U)\right),\]
	    has full measure in $U$ for every $n\geq \hat{n}$. 
	    For $z\in \mathcal E^U = \bigcap_{n\geq \hat{n}} \mathcal E^U_n$, let $n_0>0$ such that for $n\geq n_0 \geq \hat{n}$ we have $B_{d_x}(z,1/n_0)\subset U$. For $n\geq n_0$, let $j$ with
	    \[f^{-j}(z) \in \Phi^{U}_{1/n}(\mathcal Z_U)\cap \mathcal G^{U},\]
	    and $p \in \mathcal Z_U$ with $d_x(p,f^{-j}(z))<1/n$. Then $d_x(f^j(p),z)<1/n$, which implies $f^j(p)\in U$, and if $\mu^U_p(B_{d_x}(p,\delta))=0$, for $\delta$ small, then since $\mu^U_p \sim \omega_{f^{-j}(z)}$ (because $f^{-j}(z) \in \mathcal G^U$) and $z\in \mathcal G$, it follows that $\omega_z(f^{-j}(I_p))=\omega_{f^j(z)}(I_p) = 0$ and $z\in \mathcal G^u$, thus $\mu^U_z(f^{-j}(I_p))=0$. Therefore, $f^j(p)\in \mathcal Z_U$ and $z\in \Phi_{1/n}(\mathcal Z_U)$. 
	    That is, the set $\mathcal Z_U \cap \mathcal F|U(z)$ is dense in $\mathcal F|U(z)$, for almost every $z\in \mathcal E^U$. \quad \\
	    
	{\bf Claim:}
	For $x\in \mathcal E^U$, $\mathcal{C}_x:=\mathcal{F}|U(x)\setminus \mathcal{Z}_U$ is a Cantor subset in $\mathcal{F}|U(x)$.
	\begin{proof} To prove that $\mathcal{C}_x$ is a Cantor set, we will show that this set is a nowhere dense and perfect set. Therefore, as it is bounded ($\mathcal C_x \subset \mathcal F|U(x)$), we can conclude that it is a compact, nowhere dense perfect set, that is, a Cantor set. Since $\mathcal{Z}_U \cap \mathcal F|U(x)$ is a dense and open set in $\mathcal{F}|U(x)$ we have that $\mathcal{C}_x$ is a nowhere dense and closed set. 
		Now let us see that $\mathcal{C}_x$ has no isolated points. Let $y\in \mathcal{C}_x$, suppose that there exist $r>0$ with $B_{d_x}(y,r)\subset U$ and $B_{d_x}(y,r)\cap \mathcal{C}_x=\{y\}$. Since $y\in \mathcal{C}_x$, 
		\[0<\mu^U_x(B_{d_x}(y,r))=\mu^U_x(B_{d_x}(y,r)\setminus \mathcal C_x ) + \mu^U_x(\{y\}) = \mu^U_x(\{y\}),\]
		thus  $\mu^U_x(\{y\})>0$, which is a contradiction since $x\notin \mathcal A_U$.
		
		Therefore $\mathcal C_x$ is indeed a Cantor subset. 
		\end{proof}
	Thus, for almost every $x\in M$ and any local chart $U$ the conditional measures $\mu^U_x$ is supported in the Cantor set $\mathcal{C}_x$. In particular the measure $\omega_x$ on $\mathcal F(x)$ is supported on a $\sigma$-Cantor subset of the leaf. We remark that for the Claim we did not use the fact that $\Delta=\infty$, thus the same argument will work when $\Delta<\infty$. \\

	\item Case 3.2: If, on the other hand, there exists $N_0\in \mathbb{N}$ with $N_0\geq \hat{n}$ such that $\mu(\mathcal E_{N_0})=0$, then $\mu(\mathcal E^U_{N_0})=0$ and moreover $\mu(\mathcal E^U_N)=0$, for any $N \geq N_0$. Since
\[\mathcal G^U \cap \mathcal G \cap \varphi_{1/N}(\mathcal Z_U) \subset \mathcal E^U_{N_0},\]
we have
\[\mu(\Phi^U_{1/N}(\mathcal Z_U)) = \mu(\mathcal G^U \cap \mathcal G \cap \Phi_{1/N}(\mathcal Z_U))\leq \mu(\mathcal E_n^U)=0, \quad \forall N\geq N_0.\]
In particular, for almost every $x\in U$ we have
\begin{equation}\label{eq:lasteq}\mu^U_x(\Phi^U_{1/N}(\mathcal{Z}_U\cap \mathcal F|U(x)))=0.\end{equation}
As $\Phi^U_{1/N}(\mathcal{Z}_U\cap \mathcal F|U(x)))$ is an open subset of $\mathcal F|U(x)$, by \eqref{eq:lasteq} we have, for almost every $x\in U$, 
\[\Phi^U_{1/N}(\mathcal{Z}_U\cap \mathcal F|U(x)))\subset \mathcal{Z}_U\cap \mathcal F|U(x)).\]
But clearly the other continence holds, thus $\mathcal{Z}_U\cap \mathcal F|U(x) = \Phi^U_{1/N}(\mathcal{Z}_U\cap \mathcal F|U(x)))$ and this implies $\mathcal{Z}_U\cap \mathcal F|U(x)) = \mathcal F|U(x)$.
As this happens for almost every $x\in U$ we fall in contradiction with the fact that $\mu(\mathcal Z_U \cap U)=0$. Therefore this case does not happen.

	\end{itemize}

	\subsection{Case 2: $\Delta=\infty$ and $\mu(D^{\infty})=0$.}\label{deltayDinfinitos}  \quad \\
	In particular for every $U\in \mathcal U$ we must have $\mu(D^{\infty}_U)=0$ which implies $\mu(\mathcal F|U(\mathcal Z_U))=\mu(U)$.
	
	In this case we will proceed very similarly to the previous Case. For $U\in \mathcal U$, consider the sets $\mathcal E_n$ and $\mathcal E_n^U$ defined in Case 1.
	
Again, if $\mu(\mathcal E_n)=1$ for every $n\in \mathbb Z$, then there exists a full measure subset of $U$, namely $\mathcal E^U$, such that $z\in \mathcal E^U$ implies $\mathcal Z_U \cap \mathcal F|U(z)$ is dense is $\mathcal F|U(z)$. Hence, as showed by the Claim in Case 1, it follows that the support of $\mu^U_x$ is a Cantor subset of the plaque $\mathcal F|U(x)$ for almost every $x\in U$.

	Otherwise, if $\mu(\mathcal E_{N_0})=0$ for some $N_0\in \mathbb N$, then as in Case 1 we conclude that $\mathcal{Z}_U\cap \mathcal F|U(x)) = \mathcal F|U(x)$ contradicting the fact that $\mu(\mathcal Z_U \cap U)=0$, thus this case does not occur.
	
	
	
	\subsection{Case 3: $\Delta = \infty$ and $\mu(D^{\infty})=1$. } \label{sec:Case3} \quad \\ 
	Let us prove that this case cannot occur. Since $\mu(D^{\infty})=1$, there exists $U\in \mathcal U$ with $\mu(D^{\infty}_U)>0$. Since $\mu(\mathcal F|U(\Pi^{\infty}_U))=\mu(U)$, for almost every point $\bar{x}\in D^{\infty}_U$ we have
	\begin{equation}\label{eq:think}
	\mu^U_{\bar{x}}(\Pi^{\infty}_U\cap \mathcal F|U(\bar{x}))=1.\end{equation}
	Take any such typical $\bar{x}$ and consider $x\in \mathcal F|U(\bar{x}) \cap \Pi^{\infty}_U$, in particular  $B_{d_x}(x,\mathfrak r)\subset U$. Also, $\Pi^{\infty}_{x,U}\cap B_{d_x}[x,\delta]$ is closed in $\mathcal F|U(x)$ for some $\delta>0$ small, if there exists $z\in B_{d_x}[x,\delta] \setminus \Pi^{\infty}_{x,U}\cap B_{d_x}[x,\delta]$ then for some $\delta_2>0$ we have $B_{d_x}(z,\delta_2) \subset B_{d_x}[x,\delta] \setminus \Pi^{\infty}_{x,U}\cap B_{d_x}[x,\delta]$ and $\mu^U_x(B_{d_x}(z,\delta_2))=0$ by \eqref{eq:think}. But this cannot happens since this would imply $z\in \mathcal Z_U $ and, consequently, $\mathcal F|U(\mathcal Z_U) \cap D^{\infty}_U \ne \emptyset$, falling in contradiction with the definition of $D^{\infty}_U$. 
	Therefore
	\[\Pi^{\infty}_{x,U}\cap B_{d_x}[x,\delta]= B_{d_x}[x,\delta].\]
	
	Consider $0<r_0<\delta$ small enough so that $B_{d_x}(x,\mathfrak r+2\cdot r_0)\subset U$.
	By hypothesis, for $k$ with $\varepsilon_k<r_0$, by Remark \ref{metricsystem}, we can take $\lfloor r_0/\varepsilon_k\rfloor$ disjoint balls of radius $\varepsilon_k$ inside $B_{d_x}(x,r_0)$, say with center $a_1,a_2,\ldots,a_{\lfloor r_0/\varepsilon_k\rfloor}$. Then
	\[\sum_{i=1}^{\lfloor r_0/\varepsilon_k\rfloor}\mu^U_{x}(B_{d_x}(a_i,\varepsilon_k)) \leq \mu^U_x(B_{d_x}(x,r_0)) \]
	\[\Rightarrow \sum_{i=1}^{\lfloor r_0/\varepsilon_k\rfloor}\frac{\mu^U_{x}(B_{d_x}(a_i,\varepsilon_k))}{\mu^U_x(B_{d_x}(x,\mathfrak r))} \leq \frac{\mu^U_x(B_{d_x}(x,r_0))}{\mu^U_x(B_{d_x}(x,\mathfrak r))}\]
	\[\Rightarrow \sum_{i=1}^{\lfloor r_0/\varepsilon_k\rfloor}\frac{\mu^U_x(B_{d_x}(a_i,\mathfrak r))}{\mu^U_x(B_{d_x}(x,\mathfrak r))}\cdot \frac{\mu^U_{x}(B_{d_x}(a_i,\varepsilon_k))}{\mu^U_x(B_{d_x}(a_i,\mathfrak r))} \leq \frac{\mu^U_x(B_{d_x}(x,r_0))}{\mu^U_x(B_{d_x}(x,\mathfrak r))}.\]
	
	By Lemma \ref{lemma:continuousinleaf}, 
	\[w \in B_{d_x}[x,r_0] \subset F_{\mathfrak r} \mapsto \mu^U_x(B_{d_x}(w,\mathfrak r))\]
	is continuous, hence there exists $\eta>0$ such that 
	\[\frac{\mu^U_x(B_{d_x}(w,\mathfrak r))}{\mu^U_x(B_{d_x}(x,\mathfrak r))} \geq \eta,\]
	for every $w\in B_{d_x}[x,r_0]$. Therefore, since $a_i \in B_{d_x}[x,\delta]$ for all $i$, we have 
	\begin{align*}
	    \eta \cdot (2\varepsilon_k\cdot k ) \cdot \lfloor \frac{r_0}{\varepsilon_k}\rfloor & <\eta \cdot \sum_{i=1}^{\lfloor r_0/\varepsilon_k\rfloor} \frac{\mu^U_{x}(B_{d_x}(a_i,\varepsilon_k))}{\mu^U_x(B_{d_x}(a_i,\mathfrak r))} \\
	    & \leq \sum_{i=1}^{\lfloor r_0/\varepsilon_k\rfloor}\frac{\mu^U_x(B_{d_x}(a_i,\mathfrak r))}{\mu^U_x(B_{d_x}(x,\mathfrak r))}\cdot \frac{\mu^U_{x}(B_{d_x}(a_i,\varepsilon_k))}{\mu^U_x(B_{d_x}(a_i,\mathfrak r))} \\
	    & \leq \frac{\mu^U_x(B_{d_x}(x,r_0))}{\mu^U_x(B_{d_x}(x,\mathfrak r))}.
	\end{align*}

	Taking $k\rightarrow \infty$, the left side goes to infinity from where we conclude that $\mu^U_x(B_{d_x}(x,r_0)) =\infty$ falling in contradiction. Thus, this case does not occur.
		

	\subsection{Case 4: $\Delta < \infty$ and $\mu(D)=1$.} \quad \\
	We will prove that if this case occurs then for almost every $x\in U$, $U\in \mathcal U$, the conditional measure in $\mathcal{F}|U(x)$ is equivalent to the  measure $\lambda_x$ given in definition \ref{def.lambdax}.
	\begin{lemma} \label{lemma:uniformDelta} 
		The constant $\Delta$ is bounded away from zero and
		\[\mu_x\ll\lambda_x ,\]
		for $\mu$-almost every $x\in M$.
	\end{lemma}
	\begin{proof}
		Let $y\in D$. Then, for some $n_0\in \mathbb Z$ and $U\in \mathcal U$ we have $f^{n_0}(y) \in D_U$. 
		Call $x=f^{n_0}(y)$. As $x\in \mathcal F|U( \Pi_U)\setminus \mathcal F|U(\mathcal Z_U)$ we have $\mu^U_x(\mathcal F|U( \Pi_U) \cap \mathcal F|U(x))=1$. Therefore we conclude (by the same argument used in Case 3 - Section \ref{sec:Case3}) that 
		\[\Pi_{x,U}\cap B_{d_x}[x,\delta] = B_{d_x}[x,\delta],\]
		for some $\delta$ small. Consequently $\Pi_U \cap \mathcal F|U(x) \supset \{y \in \mathcal F|U(x): d_{x}(y, \partial \mathcal F|U(x))\geq \mathfrak r\}$.
		
		For any given $k \in \mathbb N^{*}$, $U\in \mathcal U$ and $x \in \Pi_U$ we have
		\begin{equation}\label{eq:all}
			\left| \frac{\mu^U_x(B_{d_x}(x,\varepsilon_k))}{2\varepsilon_{k} \cdot \mu^U_x(B_{d_x}(x,\mathfrak r))} - \Delta  \right|   \leq \frac{1}{k} .\end{equation} 
		
		Given $\varepsilon > 0$ take $k_0 \in \mathbb N$ such that $k_{0}^{-1} < \varepsilon$. Again since $\{d_x\}$ is a $\mathcal{F}$-metric system, given a constant $r>0$ we need at most $s(k)=\lfloor r/\varepsilon_{k}\rfloor+1$ points, say $a_1,a_2,...,a_{s(k)}$, to cover the ball $B_{d_x}(x,r)$ with balls of radius $\varepsilon_k$. Let $\alpha_i :=\frac{\mu^U_x(B_{d_x}(a_i,\mathfrak r))}{\mu^U_x(B_{d_x}(x,\mathfrak r))}$. Again by continuity (see Lemma \ref{lemma:continuousinleaf}) there exists $\beta>0$ such that $\alpha_i \leq \beta$ for all $i$. Therefore
		\begin{align*}
			\frac{\mu^U_x(B_{d_x}(x,r))}{\mu^U_x(B_{d_x}(x,\mathfrak r))}\leq & \sum_{i=1}^{s(k)} \frac{\mu^U_{x}(B_{d_x}(a_i,\varepsilon_k))}{\mu^U_x(B_{d_x}(x,\mathfrak r))}
			=  \sum_{i=1}^{s(k)} \alpha_i \cdot \frac{\mu^U_{x}(B_{d_x}(a_i,\varepsilon_k))}{\mu^U_x(B_{d_x}(a_i,\mathfrak r))} \\
			\leq & \beta \cdot \sum_{i=1}^{s(k)}  \frac{\mu^U_{x}(B_{d_x}(a_i,\varepsilon_k))}{\mu^U_x(B_{d_x}(a_i,\mathfrak r))}
			\leq \beta \sum_{i=1}^{s(k)} \frac{2\varepsilon_k}{k} + \Delta2\varepsilon_k \\
			= & \beta \cdot s(k)\frac{2\varepsilon_k}{k} + \beta\cdot s(k)2\varepsilon_k  \cdot \Delta.
		\end{align*}
		Since $\lim s(k)\varepsilon_k=r$  we have that $\beta \cdot s(k)\frac{\varepsilon_k}{k}$ goes to zero as $k\rightarrow \infty$, we have
		\[\frac{\mu^U_x(B_{d_x}(x,r))}{\mu^U_x(B_{d_x}(x,\mathfrak r))} \leq 2\Delta \cdot \beta \cdot r. \]
		
		Therefore $\mu^U_x\ll\lambda_x$ when restricted to $\{y \in \mathcal F|U(x): d_{x}(y, \partial \mathcal F|U(x))\geq \mathfrak r\}$ and $\Delta >0$.
		As $\mathfrak r$ can be taken to be arbitrarily small in the beginning, it follows that $\mu^U_x \ll \lambda_x$ as we wanted to show.
		
		%
		%
	\end{proof}

	Next we are able to conclude that $\mu^U_x$ is equivalent to the measure  $\lambda_x$. 
	
	\begin{lemma} \label{lemma:equallebesgue}
		For $\mu$ almost every $x\in M$
		\[\mu^U_x\sim\lambda_x.\]
	\end{lemma}
	\begin{proof}
		By Lemma \ref{lemma:uniformDelta}  we know that $\mu^U_x \ll\lambda_x$. Since $\lambda_x$ is a doubling measure we have that the Radon-Nikodyn derivative $d\mu^U_x/d\lambda_x$ exists and is given at $\lambda_x$-almost every point $y\in \{y \in \mathcal F|U(x): d_{x}(y, \partial \mathcal F|U(x))\geq \mathfrak r\}$ by
		\[\frac{d\mu^U_x}{d\lambda_x}(y) = \lim_{r\rightarrow 0}\frac{\mu^U_x(B_{d_x}(y,r))}{\lambda_x(B_{d_x}(y,r))}.\]
		In particular, by taking the limit along the subsequence $\varepsilon_k$, $k\rightarrow \infty$, we conclude that
		\[\frac{d\mu^U_x}{d\lambda_x}(y)  = \lim_{k\rightarrow \infty}\frac{\mu^U_x(B_{d_x}(y,\varepsilon_k))}{\lambda_x(B_{d_x}(y,\varepsilon_k))}, \quad \lambda_x-a.e \quad y\in \{y \in \mathcal F|U(x): d_{x}(y, \partial \mathcal F|U(x))\geq \mathfrak r\}, \]
		which implies
		\[\frac{d\mu^U_x}{d\lambda_x}(y)  =\beta(y) \cdot \Delta, \quad \lambda_x-a.e \quad y\in \{y \in \mathcal F|U(x): d_{x}(y, \partial \mathcal F|U(x))\geq \mathfrak r\}.\]
		where $\beta(y)=\mu^U_x(B_{d_x}(y,\mathfrak r))$. Since $\beta$ is  a continuous function, given any compact $I\subset \{y \in \mathcal F|U(x): d_{x}(y, \partial \mathcal F|U(x))\geq \mathfrak r\}$ we have
		
		\[\beta_1\Delta\leq\frac{d\mu_x}{d\lambda_x}(y)\leq \beta_2\Delta, \, \lambda_x \text{ a.e }  \, y\in I.\]
		Then,
		\[\beta_1\Delta\cdot \lambda_{x}\leq \mu_x, \, \lambda_x \text{ a.e }  \,y\in I.\]
		In particular, if $\mu^U_x(E)=0$ then, for any compact subset $I\subset \{y \in \mathcal F|U(x): d_{x}(y, \partial \mathcal F|U(x))\geq \mathfrak r\}$ we have $\lambda_x(E\cap I)=0$. Since $\mathcal F(x)$ may be written as a countable union of increasing compact subsets, we conclude that $\lambda_x(E \cap \{y \in \mathcal F|U(x): d_{x}(y, \partial \mathcal F|U(x))\geq \mathfrak r\})=0$. Again, since $\mathfrak r$ may be taken to be arbitrarily small we conclude that $\lambda_x(E)=0$, thus $\lambda_x\ll \mu^U_{x}$ as we wanted to show.
	\end{proof}

 \section{Proof Theorem \ref{theorem:casobonatti}}\label{sec:dicotomia}
The proof of Theorem \ref{theorem:casobonatti} employs arguments inspired in those used in \cite{BorisKatok}, with several technical differences related to the fact that the balls inside the leaves here are generated by the plaque-continuous metric system $\{d_x\}_{x \in M}$.

Additionally, instead of affine maps, we define functions on $\mathcal{F}^c(x)$, as described in \cite{BonattiZhang}, referred to as the limit center map.

Consider the set $\mathcal N=\{x\in M \, : \, \alpha(x)=\omega(x)=M\}$. Since $f$ is transitive and $\mu$ has full support, $\mathcal N$ is a full measure residual subset. The following definition is due to Bonatti-Zhang \cite{BonattiZhang}.
\begin{definition}
	 We say that a map $F:\mathcal F^c(x)\to \mathcal F^c(x)$ is a \textit{limit center map} if there exists a sequence $\{n_i\}\subset \mathbb{Z}$ with $|n_i|\to \infty$ such that $\{f^{n_i}\}$ pointwise converges to $F$.
\end{definition}
For each $x\in \mathcal N$ we consider  
\begin{align*}	
	\mathscr{L}(\mathcal F^c(x))&:=\{F: \mathcal F^c(x) \to \mathcal F^c(x): F \text{ is a limit center map}\} \text{ and }\\
	\mathscr{L}^+(\mathcal F^c(x))&:=\{F\in \mathscr{L}(\mathcal F^c(x)) \, : \, F \text{  preserves the orientation of } \mathcal F^c(x)\}.
\end{align*}
\begin{remark}\label{remark6.2}(see \cite[Proposition 4.7]{BonattiZhang} For every $x\in \mathcal N$, we have:
	\begin{itemize}
		\item If $L = \mathcal F^c(x)$ is not compact, then there is a homeomorphism $\psi_L: L\to \mathbb{R}$, such that
		\[
		\mathscr L^+(L)=\{\psi_L^{-1}\circ T_t\circ \psi_L, \, t\in \mathbb{R}\},
		\]
		where $T_t$ is the translation $T_t:\mathbb{R}\to \mathbb{R}, \, s\mapsto s+t$.
		
		\item If $L=\mathcal F^c(x)$ is compact, then there is a homeomorphism $\psi_L:L\to S^1$, such that
		\[
		\mathscr L^+(L)=\{\psi_L^{-1}\circ R_t\circ \psi_L, \, t\in S^1=\mathbb{R}/\mathbb{Z}\},
		\]
		where $R_t$ is the rotation $R_t:S^1\to S^1, \, s\mapsto s+t(\mathrm{mod}\mathbb{Z})$.

		\item  $\mathscr{L}^+(\mathcal F^c(x))$ is a group whose  action on  $\mathcal F^c(x)$ is transitive.
		\item $\mathscr{L}(\mathcal F^c(x))$ either coincides with $\mathscr{L}^+(\mathcal F^c(x))$ or is the group	generated by $\mathscr{L}^+(\mathcal F^c(x))$ and $-\textrm{Id}|{\mathcal F^c(x)}$.
		\item  If $F:\mathcal F^c(x)\to \mathcal F^c(x)$ is a limit center map having a fixed point $x\in \mathcal F^c(x)$,  if $F$ is orientation preserving, then $F$ is the identity map of $\mathcal{F}^c(x)$. This is, the action on $\mathcal F^c(x)$ given by the group $\mathscr L^+(\mathcal F^c(x))$ is free.
	\end{itemize}
\end{remark}

Before we proceed, observe that  recall that the atlas  $\mathcal U$ of $\mathcal F = \mathcal F^c$ was taken to have specific properties detailed in the beginning of Section \ref{sec:FiberedSpaces}. Furthermore, in what follows we assume that the disintegration of $\mu$ along $\mathcal F^c$ is not atomic, otherwise we already fall in one of the possibilities claimed by the statement.

For each point $x\in M$, consider $U_x \in \mathcal U$ and $V_x \subset U_x$ the closure of a neighbourhood of $x$ such that for every $y\in V_x$ we have $B_{d_y}(y,\mathfrak r)\subset U$.

\begin{lemma}\label{lemma:aux}
	For $\mu$-almost every $x\in M$ and $\mu^{U_x}_x$-almost every $y\in V_x\cap \mathcal F^c|U_x(x)$, there is a limit center map $F \in \mathscr{L}^+(\mathcal F^c(x))$ such that $F(x)=y$ and $F_*\mu_x=\mu_y$.
\end{lemma}
\begin{proof}
Recall that for every $x\in M_0$ (see \ref{defiM0}) the measure $\mu_x$ is defined on $B_{d_x}(x,\mathfrak r)$ by
\[\mu_x(J) = \mu_x^U(J | B_{d_x}(x,\mathfrak r)) = \frac{\mu_x^U(J \cap B_{d_x}(x,\mathfrak r))}{\mu_x^U(B_{d_x}(x,\mathfrak r)) } , \quad J\subset B_{d_x}(x,\mathfrak r).\]
Abusing notation we may define $x\in M_0 \mapsto \mu_x(J)$ for any Borel set $J\subset M$ by taking $\mu_x(J) := \mu_x(J\cap B_{d_x}(x,\mathfrak r))$. By Corollary \ref{cor:jointlymeasurable} the map $x \in M_0 \mapsto \mu_x$ is measurable in the sense that for any Borel subset $J \subset M$ the map $x \in M_0 \mapsto \mu_x(J\cap B_{d_x}(x,\mathfrak r))$ is measurable.

Using Lusin's Theorem we may take $\{K_i\}$ an increasing sequence of compact subsets of $M_0$ for which the map $x\mapsto \mu_x$ is continuous when restricted to each $K_i$. Once again, we emphasize that, in this case, it means that for any Borel subset $J\subset M$ the map $x \in K_i \mapsto \mu_x(J\cap B_{d_x}(x,\mathfrak r))$ is continuous. As $\mu$ is ergodic, for $\mu$-almost every point $x\in K_i$, the orbit of $x$ is dense in a full measurable subset of $K_i$. Let  $P_i=\mathcal N\cap K_i$.
	
	Let $x\in P_i$ and $y\in V_x \cap \mathcal{F}^c|U_x(x)\cap P_i$. From Remark \ref{remark6.2}, we know that there exists $F\in \mathscr{L}^+(\mathcal{F}^c(x))$ such that $F(x)=y$, and since $F$ is a limit center map, there exists a sequence ${n_k}\subset \mathbb{Z}$ with $|n_k|\to \infty$ so that the sequence $f^{n_k}|_{\mathcal{F}^c(x)}$ converges pointwise to $F$.
	
	By the continuity of $x \in P_i \mapsto\mu_x$, and since $y\in P_i$, we have that
	\[
	F_*\mu_x=\lim_{k\to \infty}f^{n_k}_*\mu_x=\lim_{k\to \infty}\mu_{f^{n_k}(x)}=\mu_{F(x)}=\mu_y
	\]
	By taking the limit set of $K_i$, we conclude that for every point $x\in P$, where $P=\lim_i P_i$, and $y\in V_x\cap \mathcal{F}^c|U_x(x)\cap P$, there exists $F\in \mathscr{L}^+(\mathcal{F}^c(x))$ such that $F(x)=y$ and $F_*\mu_x=\mu_y$
    as we wanted.


    
\end{proof}

Let $P$ be the full measure subset of points for which the previous Lemma holds. For $x \in P$ define
\[G^{U_x}_x = \{y \in V_x \cap \mathcal F^c|U_x(x): \exists \; F \in \mathcal L^+(\mathcal F^c(x)) \quad \text{with}\quad F(x) = y \quad \text{and} \quad F_*\mu_x = \mu_y \}.\]
By Lemma \ref{lemma:aux} we know that $[V_x \cap \mathcal F^c|U_x(x)] \setminus G_x^{U_x}$ has zero $\mu_x^{U_x}$-measure.

\begin{lemma}\label{lema:closedaux}
For $x\in P$, the set $G^{U_x}_x$ is closed in $\mathcal F^c|U_x(x)$.
\end{lemma}
\begin{proof}
    Assume that $\mathcal F^c(x)$ is not compact, the argument is identical for the other case.
    
    We know that $\mathcal L^+(\mathcal F^c(x))$ is a closed group as it is isomorphic to either the group of translations (see Remark \ref{remark6.2}). Furthermore, for any $J\subset M$ Borel, the map $y \in B_{d_x}(x,\mathfrak r)\cap M_0 \mapsto \mu_y(J\cap B_{d_x}(y,\mathfrak r))$ is continuous by Corollary \ref{cor:continuanafolha}. 
    
    Let $y_n \in G^{U_x}_x$ with $y_n \to y \in V_x \cap \mathcal F^c|U(x)$ and consider $F_n \in \mathcal L^+(\mathcal F^c(x))$ be given by Lemma \ref{lemma:aux} satisfying $F_n(x) = y_n$. 

    Observe that the sequence $\{F_n\}_n$ converges pointwise to a map $F\in \mathcal L^+(\mathcal F^c(x))$ with $F(x)=y$. Indeed, by Remark \ref{remark6.2}, there are real numbers $t_n$ for which :
    \[\psi_L^{-1}\circ T_{t_n}\circ \psi_L = F_n, \quad n\geq 1,\]
    where $L=\mathcal F^c(x)$. In particular,
    \[T_{t_n}(\psi_L(x)) = \psi_L(y_n).\]
    Since $\psi_L(y_n)$ converges to $\psi_L(y)$, because $\psi_L$ is a homeomorphism, it follows that $|\psi_L(x) -\psi_L(y_n)| \rightarrow |\psi_L(x) -\psi_L(y)|$. Since $T_{t_n}$ are translations we have $t_n = |\psi_L(x) -\psi_L(y_n)| \rightarrow |\psi_L(x) -\psi_L(y)|$. Consequently, the sequence of translations $\{T_{t_n}\}$ converges to the translation $T_t:\mathbb R\to \mathbb R$ with $t = |\psi_L(x) -\psi_L(y)|$. In particular $F_n$ converges to $F:= \psi_L^{-1}\circ T_t\circ \psi_L$. 
    
    Finally, by the continuity mentioned previously, we have
    \[F_*\mu_x = \lim_{n\rightarrow \infty}(F_n)_*\mu_x = \lim_{n\rightarrow \infty} \mu_{y_n} = \mu_y.\] 
    Thus, $y \in G^{U_x}_x$, showing that $G^{U_x}_x$ is closed in $\mathcal F^c|U_x(s)$ as we wanted.
\end{proof}

\begin{lemma}
For $\mu$-almost every $x\in P$, the support of $\mu^{U_x}_x$ cannot be a Cantor subset of $\mathcal F^c|U(x)$.
\end{lemma}
\begin{proof}
Assume that for some $x\in P$ the support of $\mu^{U_x}_x$ is a Cantor subset of $\mathcal F^c|U(x)$. As $V_x$ is compact, the intersection $V_x \cap \text{supp}(\mu^{U_x}_x)$ is also a Cantor subset and is the support of the restriction $\mu^{U_x}_x(\cdot | V_x \cap \mathcal F^c|U(x))$.

As we have seen, for $\mu$-almost every $x\in P$ the set $G^{U_x}_x$ is a full measure subset of $ V_x \cap \mathcal F^c|U(x))$ and which is closed by Lemma \ref{lema:closedaux}, therefore it must contains the support $V_x \cap \text{supp}(\mu^{U_x}_x)$ which is a Cantor subset.
Consider an ordering of $\mathcal F^c|U(x)$ and consider $a = \text{inf} (V_x \cap \text{supp}(\mu^{U_x}_x))$, $b=\text{sup} (V_x \cap \text{supp}(\mu^{U_x}_x)),$ and $\gamma$ the center arc connecting $a$ and $b$ inside $\mathcal F^c|U(x)$.

Let $\zeta : \gamma \to [0,1]$ given by $\zeta(a)=0$, $\zeta(b) =1$, increasing and linear with respect to $d_x$, that is, $\zeta(y) = d_x(a,y)/d_x(a,b)$. The set $\zeta(\gamma) \subset [0,1]$ is a Cantor subset, in particular there exists an strictly increasing homeomorphism $H_0 : \zeta(\gamma) \to \mathcal C_3$ with $H_0(0)=0$, $H_0(1)=1$, where $\mathcal C_3$ here denotes the standard ternary Cantor set in $[0,1]$. Now let us extend $H_0$ to a map of $[0,1]$.

First we fix some notations. Consider $I^0 = [0,1]$, $I^1_1=(1/3,2/3)=(a^1_1,b^1_1)$ and $I^m_k = (a^m_k,b^m_k)$, with $1\leq k \leq 2^{m-1}$ the open intervals removed at the $m$-th stage of the construction of the ternary Cantor set with $a^m_k < b^m_k < a^m_{k+1}$ for any $m\geq 2$ and $1\leq k \leq 2^{m-1}-1$. Therefore
\[\mathcal C_3 = I_0 \setminus  \left( \bigcup_{m=1}^{\infty} \bigcup_{k=1}^{2^{m-1}} I^m_k\right).\]
Let also $J^m_k := (H_0^{-1}(a^m_k), H_0^{-1}(b^m_k)) \subset [0,1]$, then $\zeta(\gamma) = I_0 \setminus  \left( \bigcup_{m=1}^{\infty} \bigcup_{k=1}^{2^{m-1}} J^m_k\right)$.
Consider $H:[0,1] \to [0,1]$ defined by:
\begin{itemize}
    \item the restriction of $H$ to each $J^m_k$ is the linear map given by:
    \[H|J^m_k(y) = \left(1-\frac{y-H_0^{-1}(a^m_k)}{H_0^{-1}(b^m_k)-H_0^{-1}(a^m_k)}\right)\cdot a^m_k+ \frac{y-H_0^{-1}(a^m_k)}{H_0^{-1}(b^m_k)-H_0^{-1}(a^m_k)}\cdot b^m_k\]

    \item the restriction of $H$ to $\zeta(\gamma)$ is $H_0$, that is, $H|\zeta(\gamma) = H_0$.
\end{itemize}
Observe that $H$ is continuous and strictly increasing, thus it is a strictly increasing bijection and, consequently, a homeomorphism. 

Consider now the homeomorphism $\widetilde{H} := H\circ \zeta : \gamma \to [0,1]$ and denote by $E_3$ the set of all points in $\mathcal C_3$ that are extremes of the removed intervals, that is, $E_3 = \{a^m_k, \; b^m_k : m\in \mathbb Z, \; 1\leq k \leq 2^{m-1} \}$, or equivalently,
\[E_3 = \left\{\frac{k}{3^m}, \frac{k+1}{3^m} : \;  m \geq 1, \; 0< k <3^m,\; k \quad \text{not divisible by } 3 \right\} .\]

    



Now, consider $y\in \text{supp}(\mu^{U_x}_x(\cdot | V_x))$ such that $y \in \widetilde{H}^{-1}(E_3)$, that is $y$ is the preimage of an extreme of some open interval removed from the ternary Cantor set. Therefore, taking the pre-image by $\widetilde{H}$ of an open interval of $[0,1]\setminus \mathcal C_3$ having $\widetilde{H}(y)$ as one of its extremes, we obtain an open arc $\gamma_y \subset \mathcal F^c|U(x)$ with extreme in $y$ and contained in the complement of $V_x \cap \text{supp}(\mu^{U_x}_x)$.  In particular, $\gamma_y$ has zero $\mu_y$-measure and $y$ as one of its extremes. As $y\in G^{U_x}_x$, there exists $F$ such that $F(x)=y$ and $F_*\mu_x = \mu_y$, therefore $F^{-1}(\gamma_y)$ is an open arc with zero $\mu_x$-measure and extreme in $x$, that is, $x \in \widetilde{H}^{-1}(E_3)$. Analogously, for any other $z\in G_x^{U_x}$ there exists $F_{x,z} \in \mathcal L^+(\mathcal F^c(x))$ with $F_{x,z}(x) = z$ and $(F_{x,z})_*\mu_x = \mu_z$, which by the same reasoning implies that $z \in \widetilde{H}^{-1}(E_3)$. Therefore $G_x^{U_x} = \widetilde{H}^{-1}(E_3)$ which is an absurd since $E_3 \ne \mathcal C_3$.
\end{proof}

As $\mu(P)=1$ then, the previous Lemma concludes that for $\mu$-almost every $x\in M$ the conditional measure $\mu_x^{U_x}$ cannot be supported on a Cantor subset of the plaque, which concludes the proof of Theorem \ref{theorem:casobonatti}.
	
	\vspace{0.5cm}
	
    \section*{\textbf{Acknowledgment}}

    We are deeply grateful to Federico Rodriguez Hertz (Penn State University) for directing the second author, while in a visit to Penn State University, to an argument of Katok-Kalinin \cite{BorisKatok}, which ultimately led to the proof of Theorem \ref{theorem:casobonatti}. We also thank Prof. Ali Tahzibi for hearing and providing suggestions on earlier versions of this paper. Finally, we would like to thank the referees for their careful reading and insightful comments, which greatly improved the clarity and presentation of the proofs.
	
	\section*{\textbf{Funding}}
	Gabriel Ponce has received research support from FAPESP, grants \# 2018/25624-0 and \# 2022/07762-2, and a partial grant from project 88887.310561/2018-00 - CAPES-PRINT. R\'egis Var\~ao also received a research support from FAPESP, grants \#  2016/22475-9, \# 18/13481-0 and \# 17/06463-3. Marcielis Espitia Noriega received a Ph.d fellowship from CAPES, Finance code 001. This study was also partially supported by the Coordena\c c\~ao de Aperfei\c coamento de Pessoal de N\'ivel Superior - Brasil (CAPES) and CNPq, Finance Code 001.
	
	\section*{\textbf{Declaration}}
	\textbf{Conflict of interest:} The authors declare that there are no competing interests. The authors declare that data
sharing is not applicable to this article as no datasets were generated or analyzed.
	
\bibliographystyle{plain}
\bibliography{Referencias}

\begin{thebibliography}{10}

\bibitem{Anosov1}
D.~Anosov.
\newblock Geodesic flows on closed riemannian manifolds with negative
  curvature.
\newblock {\em Proc. Steklov Inst. Math.}, 90:1--235, 1969.

\bibitem{AVW}
A.~Avila, M.~Viana, and A.~Wilkinson.
\newblock Absolute continuity, lyapunov exponents and rigidity i: geodesic
  flows.
\newblock {\em Journal of European Math. Soc.}, 17:1435–1462, 2015.

\bibitem{BP}
L.~Barreira and Y.~Pesin.
\newblock {\em Dynamics of Systems with Nonzero Lyapunov Exponents}.
\newblock Cambridge University Press, 2007.

\bibitem{bogachevI}
V.~I. Bogachev.
\newblock {\em Measure Theory I}, volume~1.
\newblock Springer-Verlag, Berlin, 2007.

\bibitem{BonattiZhang}
C.~Bonatti and J.~Zhang.
\newblock Transitive partially hyperbolic diffeomorphisms with one-dimensional
  neutral center.
\newblock {\em Sci. China Math.}, 63(9):1647--1670, 2020.

\bibitem{InfDimAna}
D.~Charalambos and C.~Kim.
\newblock {\em Infinite dimensional analysis}.
\newblock Springer, Berlin, 2006.

\bibitem{El.Pisa}
M.~Einsiedler and E.~Lindenstrauss.
\newblock Diagonal actions on locally homogeneous spaces.
\newblock In {\em Homogeneous flows, moduli spaces and arithmetic}, number~10
  in Clay Math. Proc., pages 155--241. American Math. Soc., Providence, RI,
  2010.

\bibitem{Einsiedler}
M.~Einsiedler and T.~Ward.
\newblock {\em Ergodic Theory with a View Towards Number Theory}, volume 259 of
  {\em Graduate Texts in Mathematics}.
\newblock Springer-Verlag London, 2011.

\bibitem{homburg.atomic}
A.~J. Homburg.
\newblock Atomic disintegrations for partially hyperbolic diffeomorphisms.
\newblock {\em Proceedings of the American Mathematical Society}, 2017.

\bibitem{BorisKatok}
A.~Katok and B.~Kalinin.
\newblock Measure rigidity beyond uniform hyperbolicity: invariant measures for
  cartan actions on tori.
\newblock {\em Journal of Modern Dynamics}, 1:123--146, 2007.

\bibitem{Lindenstrauss}
E.~Lindenstrauss.
\newblock Recurrent measures and measure rigidity.
\newblock In {\em Dynamics and randomness {II}}, volume~10 of {\em Nonlinear
  Phenom. Complex Systems}, pages 123--145. Kluwer Acad. Publ., Dordrecht,
  2004.

\bibitem{Mi}
J.~Milnor.
\newblock Fubini foiled: Katok's paradoxical example in measure theory.
\newblock {\em The Mathematical Intelligencer}, 19(2):30--32, 1997.

\bibitem{Ponce2}
G.~Ponce.
\newblock Ergodic properties of partially hyperbolic diffeomorphisms with
  topological neutral center.
\newblock in preparation.

\bibitem{PT}
G.~Ponce and A.~Tahzibi.
\newblock Central lyapunov exponents of partially hyperbolic diffeomorphisms on
  $\mathbb{T}^3$.
\newblock {\em Proc. Amer. Math. Soc.}, 142:3193--3205, 2014.

\bibitem{PTV}
G.~Ponce, A.~Tahzibi, and R.~Var\~ao.
\newblock Minimal yet measurable foliations.
\newblock {\em Journal of Modern Dynamics}, 8(1):93--107, 2014.

\bibitem{PossobonRodrigues}
R.~Possobon and C.~Rodrigues.
\newblock Geometric properties of disintegration of measures.
\newblock {\em arXiv preprint arXiv:2202.04511}, 2022.

\bibitem{HHUSurvey}
F.~Rodriguez~Hertz, M.~A. Rodriguez~Hertz, and R.~Ures.
\newblock A survey of partially hyperbolic dynamics.
\newblock In {\em Partially hyperbolic dynamics, laminations, and
  {T}eichm\"{u}ller flow}, volume~51 of {\em Fields Inst. Commun.}, pages
  35--87. Amer. Math. Soc., Providence, RI, 2007.

\bibitem{Ro52}
V.~A. Rohlin.
\newblock Lectures on the entropy theory of transformations with invariant
  measure.
\newblock {\em Uspehi Mat. Nauk}, 22(5 (137)):3--56, 1967.

\bibitem{RW}
D.~Ruelle and A.~Wilkinson.
\newblock Absolutely singular dynamical foliations.
\newblock {\em Comm. Math. Phys.}, 219:481--487, 2001.

\bibitem{Simmons}
D.~Simmons.
\newblock Conditional measures and conditional expectation; {R}ohlin's
  disintegration theorem.
\newblock {\em Discrete Contin. Dyn. Syst.}, 32(7):2565--2582, 2012.

\bibitem{TahzibiZhang}
A.~Tahzibi and J.~Zhang.
\newblock Disintegrations of non‐hyperbolic ergodic measures along the center
  foliation of da maps.
\newblock {\em Bulletin of the London Mathematical Society}, 55, 2022.

\bibitem{Va}
R.~Var{\~a}o.
\newblock Center foliation: absolute continuity, disintegration and rigidity.
\newblock {\em Ergodic Theory Dynam. Systems}, 36(1):256--275, 2016.

\end{thebibliography}
\end{document}